\documentclass[12pt,a4paper]{amsart}
\usepackage{times}
\usepackage{amssymb}
\usepackage{mathtools}
\usepackage{mathtext}
\usepackage{amsmath}
\usepackage{amsthm}
\usepackage[all]{xy}
\usepackage{tikz}
\usetikzlibrary{calc}
\usepgflibrary{shapes.geometric}
\usepgflibrary{shapes.misc}
\usetikzlibrary{positioning}
\usetikzlibrary{decorations}
\usetikzlibrary{decorations.pathreplacing}
\usetikzlibrary{shapes,snakes}
\usepackage{url}
\usepackage[T1]{fontenc}
\usepackage[english]{babel}
\usepackage[top=1in,bottom=1in,right=1in,left=1in]{geometry}

\usepackage[colorinlistoftodos]{todonotes}
\usepackage{xcolor}

\definecolor{sira}{HTML}{ED9090}

\usepackage{color}

\newtheorem{lemma}{Lemma}[section]
\newtheorem{proposition}[lemma]{Proposition}

\newtheorem{theorem}[lemma]{Theorem}
\newtheorem*{theorem*}{Theorem}

\theoremstyle{definition}
\newtheorem{definition}[lemma]{Definition}
\newtheorem{example}[lemma]{Example}

\newtheorem*{notation*}{Notation}

\theoremstyle{remark}
\newtheorem{remark}[lemma]{Remark}

\newcommand{\id}{\operatorname{id}\nolimits}

\newcommand{\XX}{\operatorname{X}\nolimits}
\newcommand{\ex}{\operatorname{ex}\nolimits}

\newcommand{\ZZ}{\mathbb{Z}}

\newcommand{\T}{\mathcal{T}}
\newcommand{\J}{\mathcal{J}}

\newcommand{\CC}{\mathcal{C}}

\newcommand{\A}{\mathcal{A}}

\newcommand{\Z}{\mathcal{Z}}

\begin{document}

\title{Cluster algebras of infinite rank as colimits} 
\author{Sira Gratz} 
\date{\today}

\maketitle

\begin{abstract}
We formalize the way in which one can think about cluster algebras of infinite rank by showing that every rooted cluster algebra of infinite rank can be written as a colimit of rooted cluster algebras of finite rank. Our framework is the category of rooted cluster algebras as introduced by Assem, Dupont and Schiffler. Relying on the proof of the posivity conjecture for skew-symmetric cluster algebras of finite rank by Lee and Schiffler, it follows as a direct consequence that the positivity conjecture holds for cluster algebras of infinite rank.
Furthermore, we give a sufficient and necessary condition for a ring homomorphism between cluster algebras to give rise to a rooted cluster morphism without specializations.
Assem, Dupont and Schiffler proposed the problem of a classification of ideal rooted cluster morphisms. We provide a partial solution by showing that every rooted cluster morphism without specializations is ideal, but in general rooted cluster morphisms are not ideal.

\end{abstract}

\section{Introduction}

Cluster algebras were introduced by Fomin and Zelevinsky \cite{FZ1} to provide an algebraic framework for the study of total positivity and canonical bases in algebraic groups. They are certain commutative rings with a combinatorial structure on a distinguished set of generators, called cluster variables, which can be grouped into overlapping sets of a given cardinality, called clusters. Classically, clusters are finite. However the theory can be extended to allow infinite clusters, giving rise to cluster algebras of infinite rank. Recently, cluster algebras and cluster categories of infinite rank have received more and more attention, for example in work by Igusa and Todorov (\cite{IT:cyclic} and \cite{IT.cluster}) and Hernandez and  Leclerc \cite{HL}, as well as in joint work with Grabowski \cite{GG}.

All the data needed to construct a cluster algebra is contained in an initial seed, consisting of an initial cluster and a rule, encoded in a skew-symmetrizable matrix, which allows us to obtain all other cluster variables and all the relations between them from the initial cluster. Every seed can be mutated at certain cluster variables (called exchangeable cluster variables), giving rise to new seeds with new clusters. Two seeds are called mutation equivalent if they are connected by a finite sequence of mutations of seeds and any two mutation equivalent seeds give rise to the same cluster algebra.

Assem, Dupont and Schiffler introduced the category $Clus$ of rooted cluster algebras in \cite{ADS}. The objects of $Clus$ are what can be thought of as pointed versions of cluster algebras; they are pairs consisting of a cluster algebra and a fixed initial seed. Fixing a distinguished initial seed allows for the definition of natural maps between cluster algebras, so-called rooted cluster morphisms, which are ring homomorphisms commuting with mutation and which provide the morphisms for the category $Clus$.

If the image of a rooted cluster morphism coincides with the rooted cluster algebra generated by the image of the initial seed, it is called an ideal rooted cluster morphism. Assem, Dupont and Schiffler ask in \cite[~Problem 2.12]{ADS} for a classification of ideal rooted cluster morphisms. We answer part of this question by showing that not every rooted cluster morphism is necessarily ideal in Theorem \ref{T:ideal}.
The counterexample we provide is a rooted cluster morphism with specialization, that is, some cluster variables get sent to integers. Rooted cluster morphisms without specializations are more nicely behaved and we characterize them by a necessary and sufficient combinatorial condition. As a result we show that every rooted cluster morphism without specializations is ideal (Proposition \ref{P:without specialisation implies ideal}).

Since we are interested in cluster algebras of infinite rank, we study colimits and limits in the category $Clus$. We show that the category $Clus$ is neither complete nor cocomplete, that is, limits and colimits do not in general exist. However, it has sufficient colimits to express any cluster algebra of infinite rank as a colimit of cluster algebras of finite rank, as we show in our main result Theorem \ref{T:connected colimit}.

\begin{theorem}
Every rooted cluster algebra of infinite rank can be written as a colimit of rooted cluster algebras of finite rank in the category $Clus$.
\end{theorem}

We expect this statement to be a useful tool in extending results that are known for (certain) cluster algebras of finite rank to cluster algebras of infinite rank. For example, it is a direct consequence of our result that the positivity conjecture for skew-symmetric cluster algebras as proved by Lee and Schiffler in \cite{LS} holds for cluster algebras of infinite rank.

An important source of (rooted) cluster algebras are triangulations of marked surfaces, as studied for triangulations of surfaces with finitely many marked points by Fomin, Shapiro and Thurston \cite{Fomin-Shapiro-Thurston}. Motivated by the study of cluster structures on a category of infinite Dynkin type by Holm and J\o rgensen \cite{HJ}, in a joint paper with Grabowski \cite{GG} cluster algebras arising from triangulations of the closed disc with infinitely many marked points on the boundary with one limit point have been classified. 

In a series of papers, Igusa and Todorov studied more general cluster categories of infinite Dynkin type $A$. In \cite{IT.cluster} they constructed the continuous cluster category $\CC_\pi$ of Dynkin type $A$. The category $\CC_\pi$ has a cluster structure encoded by certain countable triangulations of the closed disc with marked points lying densely on the boundary. Furthermore, in \cite{IT:cyclic}, they constructed cluster categories coming from cyclic posets -- and as a particular example discrete versions of infinite cluster categories of Dynkin type $A$ in \cite[~Section 2.4]{IT:cyclic}, where cluster structures are encoded by certain triangulations of the closed disc with a fixed discrete set of marked points on the boundary. The simplest case of these discrete infinite cluster categories, where the marked points have one limit point is just the category studied by Holm and J\o rgensen in \cite{HJ}. We show that every rooted cluster algebra arising from a countable triangulation of the closed disc can be written as a colimit of finite rooted cluster algebras of Dynkin type $A$, thus providing an algebraic interpretation for the cluster categories of infinite Dynkin type $A$.

\subsection*{Acknowledgements}
The author thanks her supervisor Thorsten Holm as well as David Ploog and Adam-Christiaan van Roosmalen for helpful discussions. The author would also like to thank Wen Chang and Bin Zhu for pointing out some relations between our work on ideal rooted cluster morphisms and their paper \cite{CZ}.

This work has been carried out in the framework of the research priority programme SPP 1388 Darstellungstheorie of the Deutsche Forschungsgemeinschaft (DFG). The author gratefully acknowledges financial support through the grant HO 1880/5-1.

\section{Rooted cluster algebras} \label{S:rooted cluster algebras}

Cluster algebras were introduced by Fomin and Zelevinsky in \cite{FZ1}. Throughout this paper we work with cluster algebras of geometric type and we consider their rooted versions, which we obtain by fixing an initial seed. Rooted cluster algebras are the objects in the category $Clus$ we want to work in, and which was introduced by Assem, Dupont and Schiffler \cite{ADS}. 

\subsection{Seeds}\label{S:Seeds from triangulations of the closed disc}

All the information we need to construct a (rooted) cluster algebra is contained in a so-called seed. Along with a distinguished subset of generators for our cluster algebra, it contains a rule that describes how a prescribed set of generators and the relations between them can be obtained. This rule can be encoded in a skew-symmetrizable integer matrix. A {\em skew-symmetrizable integer matrix} is a square integer matrix $B$ such that there exists a diagonal matrix $D$ with positive integer entries and a skew-symmetric integer matrix $S$ with $S=DB$.

\begin{definition}[{\cite[Section~1.2]{FZ2}}]\label{D:seed}
A {\em seed} is a triple $\Sigma = (\XX, \ex, B)$, where
\begin{itemize}
\item{$\XX$ is a countable set of indeterminates over $\mathbb{Z}$, i.e.\ the field $\mathcal{F}_{\Sigma} = \mathbb{Q}(x \mid  x \in \XX)$ of rational functions in $\XX$ is a purely transcendental field extension of $\mathbb{Q}$. The set $\XX$ is called the {\em cluster of $\Sigma$}.}
\item{$\ex \subseteq \XX$ is a subset of the cluster. The elements of $\ex$ are called the {\em exchangeable variables of $\Sigma$}. The elements $\XX \setminus \ex$ are called the {\em coefficients of $\Sigma$}.}
\item{$B = (b_{vw})_{v,w \in \XX}$ is a skew-symmetrizable integer matrix with rows and columns labelled by $\XX$, which is locally finite, i.e.\ for every $v \in \XX$ 
there are only finitely many non-zero entries $b_{vw}$ and $b_{uv}$. The matrix $B$ is called the {\em exchange matrix of $\Sigma$}. }
\end{itemize}
The field $\mathcal{F}_\Sigma= \mathbb{Q}(x \mid  x \in \XX)$ is  called the {\em ambient field} of the seed $\Sigma$.
Two seeds $\Sigma = (\XX, \ex, B = (b_{vw})_{v,w \in \XX})$ and $\Sigma' = (\XX',\ex',B'=(b'_{vw})_{v,w \in \XX'})$ are called {\em isomorphic}, and we write $\Sigma \cong \Sigma'$, if there exists a bijection $f \colon \XX \to \XX'$ inducing a bijection $f \colon \ex \to \ex'$ such that for all $v,w \in \XX$ we have $b_{vw} = b'_{f(v)f(w)}$.
\end{definition}

\begin{remark}\label{R:local finiteness}\label{R:countability}
The assumption of countability of the cluster $\XX$ in a seed is not necessary for any of our results to hold (up to a minor change in Theorem \ref{T:connected colimit}, cf.\ Remark \ref{R:colimit uncountable}). However, as we will see in Remark \ref{R:connected implies countable}, from a combinatorial viewpoint one does not observe any new phenomena by considering uncountable seeds. Where appropriate, we will include a short remark clarifying the situation for uncountable clusters.
\end{remark} 

Often when giving examples it is more intuitive to think of the combinatorics of a seed as encoded in a quiver instead of in a matrix. This is possible if the exchange matrix is skew-symmetric.

\begin{remark}\label{R:skew-symmetric}
If the exchange matrix $B$ of the seed $\Sigma = (\XX,\ex,B)$ is skew-symmetric, we can express it via a quiver $Q_B$. The vertices of $Q_B$ are labelled by elements in the cluster $\XX$ and there are $b_{vw}$ arrows from $v$ to $w$ whenever $b_{vw}\geq 0$. The quiver $Q_B$ is locally finite, i.e.\ there are only finitely many arrows incident with every vertex. 
For a seed whose exchange matrix is skew-symmetric by abuse of notation we will often write $\Sigma = (\XX,\ex,Q_B)$ for the seed $\Sigma = (\XX,\ex,B)$. 
\end{remark}

To any seed $\Sigma = (\XX, \ex, B)$ we can naturally associate its {\em opposite seed} $\Sigma^{op} = (\XX, \ex, -B)$, by reversing all signs in the exchange matrix $B$. If $B$ is skew-symmetric, this corresponds to reversing all arrows in the associated quiver $Q_B$ which gives rise to the opposite quiver $Q^{op}_B$.

An important source of seeds is provided by triangulations of surfaces with (possibly infinitely many) marked points. Throughout this paper we will follow the example of countable triangulations of the closed disc with marked points on the boundary. This provides a connection to the work of Holm and J\o rgensen \cite{HJ} and Igusa and Todorov (\cite[Section~2.4]{IT:cyclic} and \cite{IT.cluster}), covering cluster categories of countable rank which have combinatorial models via triangulations of the closed disc.

Let us start by defining what we mean by a triangulation of the closed disc $\overline{D}_2$. We cover the boundary $\partial \overline{D}_2 = S^1$ of the closed disc by $\mathbb{R}$ in the usual way: $e \colon \mathbb{R} \to S^1, x \mapsto e(x):= e^{ix}$.

For any two elements $a,b \in S^1$ choose a lifting $\tilde{a} \in \mathbb{R}$ of $a$ and $\tilde{b} \in \mathbb{R}$ of $b$ under the map $e$ such that $\tilde{a} \leq \tilde{b} < \tilde{a}+2\pi$. Then we denote by $[a,b]$ the image
\[
[a,b] = e([\tilde{a},\tilde{b}]).
\]
We define the open interval $(a,b)$ and the half-open intervals $[a,b)$ and $(a,b]$ analogously.

We view $\overline{D}_2 \subseteq \mathbb{R}^2$ as a topological space with the standard topology. Let $\Z \subseteq S^1$ be a subset of the boundary of $\overline{D}_2$. To rule out trivial cases, throughout we assume that any such subset contains at least two elements, i.e.\ $|\Z| \geq 2$.

\begin{definition}
An {\em arc of $\Z$} is a two-element subset of $\Z$, i.e.\ a set $\{x_0,x_1\} \subseteq \Z$ with $x_0 \neq x_1$.  
An arc $\{x_0,x_1\}$ of $\Z$ is called an {\em edge of $\Z$} if $(x_0,x_1) \cap \Z = \emptyset$ or $(x_1,x_0) \cap \Z = \emptyset$. An arc of $\Z$ that is not an edge of $\Z$ is called an {\em internal arc of $\Z$}.

Two arcs $\{x_0,x_1\}$ and $\{y_0,y_1\}$ are said to {\em cross} if either $y_0 \in (x_0,x_1)$ and $y_1 \in (x_1,x_0)$ or $y_1 \in (x_0,x_1)$ and $y_0 \in (x_1,x_0)$, i.e.\;if the straight line connecting $x_0$ and $x_1$ crosses the straight line connecting $y_0$ and $y_1$ in the closed disc.

A {\em triangulation of the closed disc with marked points $\Z $} is a maximal collection $\T$ of pairwise non-crossing arcs of $\Z$, i.e.\;a collection $\T$ of non-crossing arcs of $\Z$ such that every arc of $\Z$ that is not contained in $\T$ crosses at least one arc in $\T$. We call a triangulation $\T$ a {\em countable triangulation of the closed disc}, if the set $\T$ is countable.
\end{definition}

\begin{remark}
Note that in order for a triangulation $\T$ of the closed disc with marked points $\Z$ to be countable, the set $\Z \subseteq S^1$ does not need to be countable. Consider for example $\Z = S^1$ and the triangulation
\[
\T=\{\left\{e\left(\frac{m \pi}{2^n}\right),e\left(\frac{(m+1) \pi}{2^n}\right)\right\} \mid n \geq 0, 0 \leq m < 2^{n+1}\}
\]
of the closed disc with marked points $S^1$, where the endpoints of the arcs in $\T$ are a countable dense subset of $S^1$ (see Figure \ref{fig:standard triangulation} for a picture). Thus $\T$ is a countable triangulation of the closed disc with uncountably many marked points $\Z = S^1$. Similarly, any subset $\Z \subseteq S^1$ allows a countable triangulation of the closed disc with marked points $\Z$.
\end{remark}

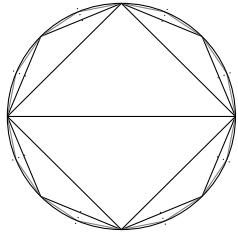
\begin{figure}
\centering{
\begin{tikzpicture}[scale = 3, font = \footnotesize, font = \sffamily, font=\sansmath\sffamily]
        \tikzstyle{every node}=[font=\small]
	\draw (0:0.5) -- (180:0.5);
	\draw (0:0.5) -- (90:0.5);
	\draw (90:0.5) -- (180:0.5);
	\draw (270:0.5) -- (180:0.5);
	\draw (0:0.5) -- (270:0.5);
	\draw[] (0:0.5) -- (45:0.5);
	\draw[] (45:0.5) -- (90:0.5);
	\draw[] (90:0.5) -- (135:0.5);
	\draw[] (135:0.5) -- (180:0.5);
	\draw[] (225:0.5) -- (180:0.5);
	\draw[] (225:0.5) -- (270:0.5);
	\draw[] (270:0.5) -- (315:0.5);
	\draw[] (0:0.5) -- (315:0.5);

	\draw[color = black!50] (0:0.5) -- (22.5:0.5);
	\draw[color = black!50] (22.5:0.5) -- (45:0.5);
	\draw[color = black!50] (45:0.5) -- (67.5:0.5);
	\draw[color = black!50] (67.5:0.5) -- (90:0.5);
	\draw[color = black!50] (90:0.5) -- (112.5:0.5);
	\draw[color = black!50] (112.5:0.5) -- (135:0.5);
	\draw[color = black!50] (135:0.5) -- (157.5:0.5);
	\draw[color = black!50] (157.5:0.5) -- (180:0.5);
	\draw[color = black!50] (0:0.5) -- (-22.5:0.5);
	\draw[color = black!50] (-22.5:0.5) -- (-45:0.5);
	\draw[color = black!50] (-45:0.5) -- (-67.5:0.5);
	\draw[color = black!50] (-67.5:0.5) -- (-90:0.5);
	\draw[color = black!50] (-90:0.5) -- (-112.5:0.5);
	\draw[color = black!50] (-112.5:0.5) -- (-135:0.5);
	\draw[color = black!50] (-135:0.5) -- (-157.5:0.5);
	\draw[color = black!50] (-157.5:0.5) -- (-180:0.5);

        	\draw (0,0) circle(0.5cm);
	
	\draw[dotted] (22:0.455) -- (22:0.52);
	\draw[dotted] (67:0.455) -- (67:0.52);
	\draw[dotted] (112:0.455) -- (112:0.52);
	\draw[dotted] (157:0.455) -- (157:0.52);
	\draw[dotted] (202:0.455) -- (202:0.52);
	\draw[dotted] (247:0.455) -- (247:0.52);
	\draw[dotted] (292:0.455) -- (292:0.52);
	\draw[dotted] (337:0.455) -- (337:0.52);
  \end{tikzpicture}
}
\caption{A countable triangulation of the closed disc with uncountably many marked points}
\end{figure}\label{fig:standard triangulation}

\begin{remark}
An edge of a subset $\Z \subseteq S^1$ crosses no other arcs of $\Z$. Thus by definition every triangulation of the closed disc with marked points $\Z \subseteq S^1$ must contain all edges of $\Z$. Note that the set of edges can be empty, for example if we have $\Z = S^1$.
\end{remark}

To any countable triangulation of the closed disc we can associate a seed, via the same method that has been introduced by Fomin, Shapiro and Thurs\-ton \cite[Definition~4.1 and Section~5]{Fomin-Shapiro-Thurston} for finite triangulations of surfaces.

\begin{definition}\label{seed of a triangulation}
Let $\T$ be a countable triangulation of the closed disc with marked points $\Z \subseteq S^1$. The {\em seed $\Sigma_{\T}$ associated to $\T$} is the triple $\Sigma_{\T} = (\T, \ex_{\T}, Q_{\T})$ defined as follows.
\begin{itemize}
\item{The elements in the cluster are labelled by the arcs in $\T$.}
\item{An arc $\{x_0,x_1\} \in \T$ is called {\em exchangeable in $\T$}, if it is the diagonal of a quadrilateral in $\T$, i.e.\;if there exist vertices $y_0,y_1 \in \Z$ with $y_0 \in (x_0,x_1)$ and $y_1 \in (x_1,x_0)$ such that $\{x_0,y_0\}, \{y_0,x_1\}, \{x_1,y_1\}$ and  $ \{y_1,x_0\}$ lie in $\T$. The exchangeable variables $\ex_\T$ are labelled by exchangeable arcs in $\T$. 
}
\item{The exchange matrix of $\Sigma_\T$ is skew-symmetric and we express it via the quiver $Q_\T$: The vertices of $Q_{\T}$ are labelled by the arcs in $\T$, and for $\{x_0,x_1\}, \{y_0,y_1\} \in \T$ there is an arrow $\{x_0,x_1\} \to \{y_0,y_1\}$ in $Q_{\T}$ if and only if the arcs $\{x_0,x_1\}$ and $\{y_0,y_1\}$ are sides of a common triangle in $\T$ and $\{y_0,y_1\}$ lies in a clockwise direction from $\{x_0,x_1\}$:
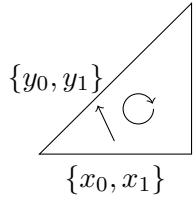
\begin{center}
\begin{tikzpicture}
\tikzstyle{every node}=[font=\small]
\draw (0,0) -- node[below]{$\{x_0,x_1\}$} (2,0) -- (2,2) -- node[left] {$\{y_0,y_1\}$} (0,0) -- cycle;
\draw[->] (10:1) -- (40:1);
\draw[->] (1.5,0.6) arc (0:0.2) -- (-340:0.2);
\end{tikzpicture}
\end{center}
}
\end{itemize}
\end{definition}

Because every arc in $\T$ is the side of at most two triangles in $\T$, the quiver $Q_{\T}$ is locally finite and the seed $\Sigma_{\T}$ associated to a triangulation $\T$ of the closed disc is indeed a seed in the sense of Definition \ref{D:seed} in light of Remark \ref{R:skew-symmetric}. 

\begin{remark}\label{R:exchangeable arcs}
An exchangeable arc in a triangulation $\T$ of the closed disc is always internal, as every edge is adjacent to at most one triangle in $\T$ and hence cannot be the diagonal of a quadrilateral in $\T$. However, not every internal arc is necessarily exchangeable. Consider for example the subset 
\[\Z = \{e\left({\frac{\pi}{k}}\right)|k \in \mathbb{Z}\setminus \{0\}\}\subseteq S^1\] 
which has exactly one limit point at $1$, and the triangulation $\T$ of the closed disc with marked points $\Z$ whose internal arcs are given by
\begin{align*}
\T_{int} =& \{\{e\left({\frac{\pi}{2}}\right),e\left({\frac{\pi}{k}}\right)\}|k \in \mathbb{Z}_{> 3}\}  \cup \{\{e\left({-\frac{\pi}{2}}\right),e\left({-\frac{\pi}{k}}\right)\}|k \in \mathbb{Z}_{>3}\} \\ 
&\cup \{\{e\left({\frac{\pi}{2}}\right),e\left({-\frac{\pi}{2}}\right)\}\}
\end{align*}
 (see Figure \ref{fig:split fountain}), i.e.\ $\T$ consists of the union of $\T_{int}$ and all edges of $\Z$.
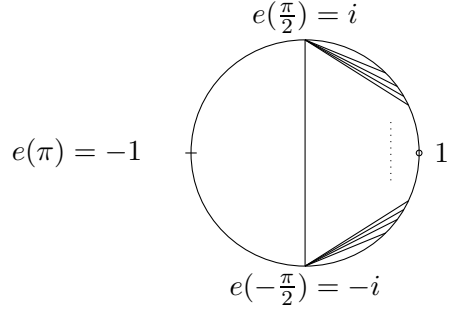
\begin{figure}
\centering{
\begin{tikzpicture}[scale = 1.5, font = \footnotesize, font = \sffamily, font=\sansmath\sffamily]
\tikzstyle{every node}=[font=\small]
\draw (0:0) circle (1);

\draw (90:1) -- (45:1);
\draw (90:1) -- (36:1);
\draw (90:1) -- (30:1);
\draw (90:1) -- (25:1);
\draw[dotted] (20:0.8) -- (-20:0.8);

\draw (-90:1) -- (-45:1);
\draw (-90:1) -- (-36:1);
\draw (-90:1) -- (-30:1);
\draw (-90:1) -- (-25:1);

\draw (90:1) -- (-90:1);
\draw (180:0.95) -- (180:1.05);
\node (a) at (1.2,0){$1$};
\node (a) at (90:1.2){$e({\frac{\pi}{2}}) = i$};
\node (a) at (-90:1.2){$e({-\frac{\pi}{2}})=-i$};
\node (a) at (180:2){$e({\pi})=-1$};
\draw (0:1) circle (0.025);

\end{tikzpicture}
}
\caption{Example of a triangulation with non-exchangeable  internal arc $\{-i,i\}$}
\label{fig:split fountain}
\end{figure}
The arc $\{e({\frac{\pi}{2}}),e({-\frac{\pi}{2}})\} \in \T$ is internal. However, it is not exchangeable: If it was, then it would have to be contained in a quadrilateral in $\T$, so there would exist a $z \in (e({-\frac{\pi}{2}}),e({\frac{\pi}{2}})) \cap \Z$ with $\{e({\frac{\pi}{2}}),z\},\{z,e({-\frac{\pi}{2}})\} \in \T$. However, if $z \in (1,e({\frac{\pi}{2}}))$ then the arc $\{z,e({-\frac{\pi}{2}})\}$ intersects infinitely many of the arcs in $\{\{e({\frac{\pi}{2}}),e({\frac{\pi}{k}})\}|k \in \mathbb{Z}_{> 2}\} \subseteq \T$ and otherwise, if $z \in (e({-\frac{\pi}{2}}),1)$, the arc $\{e({\frac{\pi}{2}}),z\}$ intersects infinitely many of the arcs in $\{\{e({-\frac{\pi}{2}}),e({\frac{\pi}{k}})\}|k \in \mathbb{Z}_{< 2}\} \subseteq \T$. This leads to a contradiction, since arcs in $\T$ have to be pairwise non-crossing.
\end{remark}

\subsection{Mutation}\label{S:mutation}

A seed $\Sigma = (\XX, \ex, B)$ contains all of the data that is needed to construct the associated (rooted) cluster algebra. In order to actually obtain generators for the cluster algebra, a combinatorial process, which is called mutation, is applied. The information needed to perform mutation is encoded in the exchange matrix $B$.

\begin{definition}[{\cite[Definition~1.1]{FZ2}}]\label{D:mutation of a seed}
Let $\Sigma = (\XX, \ex, B)$ be a seed and let $x \in \ex$ be an exchangeable variable of $\Sigma$. We denote the {\em mutation of $\Sigma$ at $x$} by  
$\mu_x(\Sigma) =  (\mu_{x}(\XX), \mu_{x}(\ex), \mu_{x}(B))$.
It is defined by the following data.
\begin{itemize}
\item{For any $y \in X$ the mutation of $y$ at $x$ is defined by
\begin{align*}
\mu_{x}(y) = y, \text{ if } y \neq x \\
\end{align*}
and
\begin{eqnarray}\label{exchange relation}
\mu_{x}(x) = \frac{\prod\limits _{v \in \XX \colon b_{xv}>0} v^{b_{xv}} + \prod\limits _{v \in \XX \colon b_{xv}<0} v^{-b_{xv}}}{x} \in \mathcal{F}_{\Sigma}.
\end{eqnarray}
The equations of the form \eqref{exchange relation} are called {\em exchange relations}. The cluster, respectively the exchangeable variables, of the seed $\mu_{x}(\Sigma)$ thus are
\begin{align*}\mu_{x}(\XX) &= \{\mu_{x}(y)|y\in \XX\} = (\XX \setminus x) \cup \mu_{x}(x) \text{ and }\\ \mu_{x}(\ex) &= \{\mu_{x}(y)|y\in \ex\} = (\ex \setminus x) \cup \mu_{x}(x).\end{align*}
}
\item{The matrix $\mu_{x}(B) = (\tilde{b}_{\tilde{v}\tilde{w}})_{\tilde{v},\tilde{w} \in \mu_x(\XX)}$ is given by {\em matrix mutation of $B$ at $x$}: For $\tilde{v} = \mu_x(v)$ and $\tilde{w} = \mu_x(w)$ set
\[
\tilde{b}_{\tilde{v}\tilde{w}} = \mu_x(b_{vw})=\begin{cases}
-b_{vw} \text{ if $v = x$ or $w = x$,} \\
b_{vw}+\frac{1}{2}(|b_{vx}|b_{xw}+b_{vx}|b_{xw}|), \text{ otherwise.}
\end{cases}
\]
}
\end{itemize} 
\end{definition}

\begin{remark}\label{R:involutive}
The following facts are well-known and straightforward to check.
\begin{itemize}
\item[$(1)$]
{
Mutation is involutive, i.e.\ for a seed $\Sigma = (\XX, \ex, B)$ and any $x \in \ex$ we have $\mu_{\mu_x(x)} \circ \mu_{x}(\Sigma) = \Sigma$. 
}
\item[$(2)$]{
Let $\Sigma = (\XX,\ex,B)$ be a seed and let $x \in \ex$. The cluster $\mu_{x}(\XX)$ of the seed $\mu_{x}(\Sigma)$ is a transcendence basis of the ambient field $\mathcal{F}_\Sigma = \mathbb{Q}(X)$ of $\Sigma$.}
\end{itemize}
\end{remark}

In the case where the exchange matrix $B$ is skew-symmetric, mutation of $B$ corresponds to quiver mutation of the associated quiver $Q_B$, where mutation of the quiver $Q_B$ at a vertex $v$ of $Q_B$ is denoted by $\mu_v(Q_B) := Q_{\mu_v(B)}$.

Consider our standard example of a seed $\Sigma_\T = (\XX_\T,\ex_\T,Q_\T)$ associated to a countable triangulation $\T$ of the closed disc with marked points $\Z \subseteq S^1$. Geometrically, mutation of $\Sigma_\T$ at an exchangeable variable in $\ex_{\T}$ can be represented by a so-called diagonal flip of $\T$. Every exchangeable arc $\{x_0,x_1\} \in \ex_\T$ is the diagonal of a unique quadrilateral with vertices $x_0,x_1,x'_0$ and $x'_1$ in $\Z$, whose sides $\{x_0,x'_0\}$, $\{x'_0,x_1\}$, $\{x_1,x'_1\}$ and $\{x'_1,x_0\}$ are all contained in $\T$. 
The {\em diagonal flip of $\T$ at $\alpha = \{x_0,x_1\}$} is the map $f_\alpha \colon \T \to (\T \setminus \alpha) \cup \overline{\alpha}$ which replaces the arc $\alpha$ in $\T$ by the arc $\overline{\alpha} = \{x'_0,x'_1\}$ of $\Z$ and leaves all other arcs invariant.

It is well-known for finite triangulations of the closed disc that for any exchangeable arc $\alpha \in \T$ we have $\mu_{\alpha}(Q_{\T}) = Q_{f_{\alpha}(\T)}$. Since mutations of quivers and diagonal flips are defined locally, only a finite subquiver of $Q_\T$ is affected by the mutation at $\alpha$: This is the full subquiver consisting of $\alpha$ and those vertices of $Q_\T$ that are labelled by the arcs of the unique quadrilateral in $\T$ that has $\alpha$ as a diagonal. Therefore the equality $\mu_{\alpha}(Q_{\T}) = Q_{f_{\alpha}(\T)}$ remains true for infinite triangulations.

\subsection{Rooted cluster algebras}\label{S:Rooted cluster algebras associated to triangulations of the closed disc}

Mutation of a seed at any exchangeable variable in its cluster yields another seed, which again can be mutated at any exchangeable variable in its respective cluster. Thus we can successively mutate a seed $\Sigma$ along what are called $\Sigma$-admissible sequences. Mutation along all possible $\Sigma$-admissible sequences will provide a prescribed set of generators of the cluster algebra associated to the seed $\Sigma$, the definition of which we will recall in this section.

\begin{definition}[{\cite[Definition~1.3]{ADS}}]\label{D:admissible sequence}
Let $\Sigma = (\XX, \ex, B)$ be a seed. For $l \geq 1$ a sequence $(x_1, \ldots, x_l)$ is called {\em $\Sigma$-admissible} if $x_1 \in \ex$ and for every $2 \leq k \leq l$, we have $x_k \in \mu_{x_{k-1}}\circ \ldots \circ \mu_{x_1}(\ex)$. The empty sequence of length $l = 0$ is $\Sigma$-admissible for every seed $\Sigma$ and mutation of $\Sigma$ along the empty sequence leaves $\Sigma$ invariant. We denote by
\[
\mathrm{Mut}(\Sigma) = \{\mu_{x_l} \circ \ldots \circ \mu_{x_1}(\Sigma)\mid l \geq 0, (x_1, \ldots, x_l) \text{ $\Sigma$-admissible}\}
\]
the set of all seeds which can be reached from $\Sigma$ by iterated mutation along $\Sigma$-admissible sequences and call it the {\em mutation class of $\Sigma$}.
\end{definition}

Since mutation is involutive (see Remark \ref{R:involutive} (1)), it is clear that mutation along $\Sigma$-admissible sequences induces an equivalence relation on seeds, where two seeds $\Sigma$ and $\Sigma'$ are {\em mutation equivalent} if and only if there exists a $\Sigma $-admissible sequence $(x_1, \ldots, x_l)$ with $\mu_{x_l} \circ \ldots \circ \mu_{x_1}(\Sigma) = \Sigma'$. The mutation class of a seed $\Sigma$ is thus really an equivalence class. Analogously, mutation equivalence of locally finite, skew-symmetrizable exchange matrices is defined.

\begin{remark} \label{R:coefficients in every seed}
It is a direct consequence of Definition \ref{D:mutation of a seed} that if two seeds $\Sigma$ and $\Sigma'$ are mutation equivalent, then the coefficients of $\Sigma$ are precisely the coefficients of $\Sigma'$. Further, by Remark \ref{R:involutive} (2), any two mutation equivalent seeds  give rise to the same ambient field $\mathcal{F}_\Sigma = \mathcal{F}_{\Sigma'}$.
\end{remark}

By mutating a seed $\Sigma$ along all possible $\Sigma$-admissible sequences we obtain the mutation class $\mathrm{Mut}(\Sigma)$ of $\Sigma$ and with it a collection of overlapping clusters. Let $P(\mathcal{F}_\Sigma)$ denote the powerset  of the ambient field $\mathcal{F}_\Sigma$, 
and let \[cl_{\Sigma} \colon \mathrm{Mut}(\Sigma) \to P(\mathcal{F}_\Sigma), (\tilde{\XX}, \tilde{\ex}, \tilde{B}) \mapsto \tilde{\XX}\] be the map assigning to each seed in the mutation class of $\Sigma$ its cluster. We now define the cluster algebra associated to a given seed $\Sigma$. The original definition for cluster algebras of finite rank is given by Fomin and Zelevinsky in \cite[Definition~2.3]{FZ1} and rooted cluster algebras are defined in \cite[Definition~1.4]{ADS}.

\begin{definition}\label{D:rooted cluster algebra}\label{D:cluster algebra}
Let $\Sigma$ be a seed. The {\em cluster algebra associated to $\Sigma$} is the $\ZZ$-subalgebra of its ambient field $\mathcal{F}_{\Sigma}$ given by
\[
\A(\Sigma) = \ZZ[x\mid x \in cl_{\Sigma}(\mathrm{Mut}(\Sigma))] \subseteq \mathcal{F}_{\Sigma}.
\]
The {\em rooted cluster algebra with initial seed $\Sigma$} is the pair $(\A(\Sigma), \Sigma)$, where $\A(\Sigma)$ is the cluster algebra associated to $\Sigma$.
The elements of $cl_{\Sigma}(\mathrm{Mut}(\Sigma))$ are called the {\em cluster variables} and the coefficients of $\Sigma$ are called the {\em coefficients} of the rooted cluster algebra $(\A(\Sigma),\Sigma)$. We call the rooted cluster algebra $(\A(\Sigma),\Sigma)$ {\em skew-symmetric}, if the matrix $B$ is skew-symmetric. The {\em rank} of the rooted cluster algebra $(\A(\Sigma),\Sigma)$ is defined as the cardinality of the cluster of $\Sigma$.
\end{definition}

\begin{remark}
Traditionally, the rank of a cluster algebra $\A(\Sigma)$ is defined as the cardinality of the set of exchangeable variables of $\Sigma$, while we define it as the cardinality of the cluster of $\Sigma$. The main point of interest in this paper are cluster algebras of infinite rank, and when we talk about those we explicitely want to include cluster algebras associated to seeds with infinitely many coefficients but only finitely many exchangeable variables.
\end{remark}

\begin{example}
For a seed $\Sigma = (X, \emptyset, B)$ with no exchangeable cluster variables, we have $\mathrm{Mut}(\Sigma) = \{ \Sigma \}$ and the cluster algebra $\A(\Sigma)$ is isomorphic to the polynomial algebra $\mathbb{Z}[x\mid x \in X]$. The empty seed $\Sigma_{\emptyset} = (\emptyset, \emptyset, \emptyset)$ gives rise to the cluster algebra $\A(\Sigma_{\emptyset}) \cong \mathbb{Z}$.
\end{example}

Two seeds in the same mutation class give rise to the same cluster algebra. 
However, they do not give rise to the same rooted cluster algebra. We can think of rooted cluster algebras as pointed versions of cluster algebras.

\begin{example}\label{E:Dynkin type A}
If $\T$ is a finite triangulation of the closed disc with marked points $\Z \subseteq S^1$ of cardinality $|\Z| = n+3$ for an $n \geq 1$, then the ring $\A(\Sigma_{\T})$ is a cluster algebra of Dynkin type $A_n$, i.e.\ it is skew-symmetric and the full subquiver of the exchange quiver of $\Sigma_\T$ consisting of the vertices associated to the exchangeable variables of $\Sigma_\T$ is mutation equivalent to an orientation of the Dynkin diagram $A_n$. We say that the cluster algebra $\A(\Sigma_\T)$, respectively the rooted cluster algebra $(\A(\Sigma_\T), \Sigma_\T)$ is {\em of finite Dynkin type $A$}.
\end{example}

\section[Rooted cluster morphisms and the category of rooted cluster algebras]{Rooted cluster morphisms and the category of rooted cluster algebras
\sectionmark{Rooted cluster morphisms}
}\label{S:Category of rooted cluster algebras}
\sectionmark{Rooted cluster morphisms}

When working with cluster algebras, it is natural to wonder what a ``morphism of cluster algebras'' should be. Intuitively we want such maps to be ring homomorphisms commuting with mutation. In \cite[Section~1.2]{FZ2}, Fomin and Zelevinsky considered what they called strong isomorphisms of cluster algebras. These are isomorphisms of rings between cluster algebras that map each seed to an isomorphic seed. This idea was generalized by Assem, Schiffler and Shramchenko in \cite{ASS} via the notion of cluster automorphisms. 
A cluster automorphism is a ring automorphism of a cluster algebra which sends a distinguished seed $\Sigma$ to another seed $f(\Sigma)$ in the mutation class of $\Sigma$, such that $f$ commutes with mutation at every variable in the two clusters. Again, only ring homomorphisms between isomorphic rings are considered. Furthermore, only cluster algebras without coefficients are considered and cluster automorphisms always bijectively map clusters to clusters: There is no way to ``delete'' cluster variables.

\subsection{Rooted cluster morphisms}

In \cite{ADS} Assem, Dupont and Schiffler introduced the notion of rooted cluster morphisms. Passing from cluster algebras to rooted cluster algebras by fixing an initial seed allows for a rigorous definition of what one means for a ring homomorphism between not necessarily ring isomorphic cluster algebras to commute with mutation.

\begin{definition}[{\cite[Definition~2.1]{ADS}}]\label{D:biadmissible}
Let $\Sigma$ and $\Sigma'$ be seeds and let $f \colon \A(\Sigma) \to \A(\Sigma')$ be a map between their associated cluster algebras. A $\Sigma$-admissible sequence $(x_1, \ldots, x_l)$ whose image $(f(x_1), \ldots, f(x_l))$ is $\Sigma'$-admis\-sible is called {\em $(f,\Sigma,\Sigma')$-biadmissible}. 
\end{definition}

\begin{definition}[{\cite[Definition~2.2]{ADS}}]\label{D:rooted cluster morphism}
Let \[\Sigma = (\XX,\ex,B) \quad \text{and} \quad \Sigma'=(\XX',\ex',B')\] be seeds and let $(\A(\Sigma),\Sigma)$ and $(\A(\Sigma'),\Sigma')$ be the corresponding rooted cluster algebras, see Definition \ref{D:rooted cluster algebra}. A {\em rooted cluster morphism} from $(\A(\Sigma),\Sigma)$ to $(\A(\Sigma'),\Sigma')$ is a ring homomorphism $f \colon \A(\Sigma) \to \A(\Sigma')$ of unital rings, i.e.\ a ring homomorphism with $f(1) = 1$, satisfying the following conditions:
\begin{itemize}
\item[CM1]{$f(\XX) \subseteq \XX' \cup \; \ZZ$.}
\item[CM2]{$f(\ex) \subseteq \ex' \cup \; \ZZ$.}
\item[CM3]{The homomorphism $f$ commutes with mutation along $(f,\Sigma,\Sigma')$\--bi\-ad\-mis\-sible sequences, i.e.\ for every $(f,\Sigma,\Sigma')$-biadmissible sequence $(x_1,\ldots, x_l)$ we have
\[
f(\mu_{x_l} \circ \ldots \circ \mu_{x_1}(y)) = \mu_{f(x_l)} \circ \ldots \circ \mu_{f(x_1)}(f(y))
\]
for all $y \in \XX$ with $f(y) \in \XX'$.
} 
\end{itemize}
\end{definition}

From now on by abuse of notation we write $\A(\Sigma)$ for the rooted cluster algebra $(\A(\Sigma), \Sigma)$. 

\begin{remark}
Every cluster automorphism in the sense of Assem, Schiffler und Shramchenko \cite{ASS} can be viewed as a rooted cluster morphism from a skew-symmetric rooted cluster algebra $\A(\Sigma)$ of finite rank and without coefficients to itself (where skew-symmetry, finite rank and having no coefficients are the assumptions in \cite{ASS} for the definition of a cluster automorphism).
Thus rooted cluster morphisms really provide a generalization of the concept of cluster automorphisms.
\end{remark}

The following example includes some of the more interesting things that can happen with rooted cluster morphisms: Firstly, they may exist between non-isomorphic rings, further we may ``delete'' cluster variables by sending them to integers and we may ``defreeze'' coefficients by sending them to exchangeable cluster variables. From now on, when we consider a seed with a skew-symmetric exchange matrix pictured as a quiver, we will mark vertices associated to coefficients with squares. 

\begin{example}\label{E:rooted cluster morphism}
Consider the seeds 
\[
\Sigma = (\{x_1,x_2,x_3\}, \{x_2,x_3\}, 
\begin{tikzpicture}[baseline={([yshift={-\ht\strutbox}]current bounding box.north)}, scale=2.5,cap=round,>=latex]
        \tikzstyle{every node}=[font=\small]
\node[rectangle, draw] (1) at (0,0) {$x_1$};
\node (2) at (1,0) {$x_2$};
\node (3) at (2,0) {$x_3$};

\draw[->] (1) -- (2);
\draw[->] (3) -- (2);
\end{tikzpicture})
\] and \[\Sigma' = (\{y_1,y_2\}, \{y_1,y_2\}, \begin{tikzpicture}[baseline={([yshift={-\ht\strutbox}]current bounding box.north)}, scale=2.5,cap=round,>=latex]
        \tikzstyle{every node}=[font=\small]
\node (1) at (0,0) {$y_1$};
\node (2) at (1,0) {$y_2$};

\draw[->] (1) -- (2);
\end{tikzpicture}).\] The associated cluster algebras are as rings isomorphic to 
\begin{align*}
\A(\Sigma) & \cong  \mathbb{Z}[x_1, x_2, x_3,\frac{x_1x_3+1}{x_2}, \frac{x_2+1}{x_3},\frac{x_1x_3+x_2+1}{x_2x_3}]\\
\A(\Sigma') & \cong  \mathbb{Z}[y_1, y_2,\frac{y_1+1}{y_2}, \frac{y_2+1}{y_1},\frac{y_1+y_2+1}{y_1y_2}].
\end{align*}
Consider the ring homomorphism $f \colon \A(\Sigma) \to \A(\Sigma')$ we obtain from the projection  of $x_3$ to $1$, which acts on the cluster variables of $\Sigma$ as $x_i \mapsto y_i$ for $i = 1,2$ and $x_3 \mapsto 1$.  This ring homomorphism satisfies axioms CM1 and CM2 by definition. The only exchangeable cluster variable in $\Sigma$ whose image is exchangeable in $\Sigma'$ is $x_2$ with $f(x_2)=y_2$, so the first entry of every $(f,\Sigma,\Sigma')$-biadmissible sequence has to be $x_2$. 
We have
\[
f(\mu_{x_2}(x_2)) = f\left(\frac{x_1x_3 + 1}{x_2}\right) = \frac{y_1+1}{y_2} = \mu_{y_2}(y_2) = \mu_{f(x_2)}(f(x_2))
\]
and, since $f(x_i) \neq f(x_2)$ for $i = 1,3$, we have \[f(\mu_{x_2}(x_i)) = f(x_i) = \mu_{f(x_2)}(f(x_i)).\]
Furthermore, the only exchangeable cluster variable in $\mu_{x_2}(\Sigma)$ whose image is exchangeable in $\mu_{y_2}(\Sigma')$ is  $\mu_{x_2}(x_2)$ with $f(\mu_{x_2}(x_2)) = \mu_{y_2}(y_2)$, so all $(f,\Sigma,\Sigma')$-biadmissible sequences have alternating entries $x_2$ and $\mu_{x_2}(x_2)$. Since mutation is involutive (see Remark \ref{R:involutive}), the ring homomorphism $f$ commutes with mutation along any of these sequences. Thus axiom CM3 is satisfied and $f$ is a rooted cluster morphism.
\end{example}

The following proposition shows that the conditions for a map $f \colon \A(\Sigma) \to \A(\Sigma')$ to be a rooted cluster morphism are preserved under mutation along biadmissible sequences. 

\begin{proposition}\label{P:rooted cluster morphism biadmissible}
Let $\Sigma$ and $\Sigma'$ be seeds and let $f \colon \A(\Sigma) \to \A(\Sigma')$ be a rooted cluster morphism. Then for every $(f,\Sigma,\Sigma')$-biadmissible sequence $(x_1, \ldots, x_l)$, the map $f$ induces a rooted cluster morphism $f \colon \A(\tilde{\Sigma}) \to \A(\tilde{\Sigma}')$ between the rooted cluster algebras with initial seeds $\tilde{\Sigma} = \mu_{x_l} \circ \ldots \circ \mu_{x_1}(\Sigma)$ and $\tilde{\Sigma}' = \mu_{f(x_l)} \circ \ldots \circ \mu_{f(x_1)}(\Sigma')$.
\end{proposition}

\begin{proof}
Because $\Sigma$ and $\tilde{\Sigma}$, respectively $\Sigma'$ and $\tilde{\Sigma}'$, are mutation equivalent, we have $\A(\Sigma) = \A(\tilde{\Sigma})$ and $\A(\Sigma') = \A(\tilde{\Sigma}')$ as algebras, so $f \colon \A(\tilde{\Sigma}) \to \A(\tilde{\Sigma}')$ is well-defined as a ring homomorphism. Let ${\Sigma} = ({\XX}, {\ex}, {B})$ and ${\Sigma}' = ({\XX}', {\ex}', {B}')$ and let $\tilde{\Sigma} = (\tilde{\XX}, \tilde{\ex}, \tilde{B})$ and $\tilde{\Sigma}' = (\tilde{\XX}', \tilde{\ex}', \tilde{B}')$. Then every element $\tilde{x}$ of $\tilde{\XX}$ (respectively of $\tilde{\ex}$) is of the form $\tilde{x} = \mu_{x_l} \circ \ldots \circ \mu_{x_1}(x)$ for an $x \in \XX$ (respectively $x \in \ex$). If $f(x) \in \mathbb{Z}$, then because  $(x_1, \ldots, x_l)$ is $(f,\Sigma,\Sigma')$-biadmissible we have $x \neq x_i$ for all $1 \leq i \leq l$. Thus $\tilde{x} = x$ and $f(\tilde{x}) = f(x) \in \mathbb{Z}$.
On the other hand, if $f(x) \notin \mathbb{Z}$, then by axiom CM1 (respectively CM2) we have $f(x) \in \XX'$ (respectively $f(x) \in \ex'$). Thus by axiom CM3 for $f \colon \A(\Sigma) \to \A(\Sigma')$ we have
\[
f(\tilde{x}) = \mu_{f(x_l)} \circ \ldots \circ \mu_{f(x_1)}(f(x))
\]
which lies in $\tilde{\XX}'$ (respectively in $\tilde{\ex}'$). Thus $f \colon \A(\tilde{\Sigma}) \to \A(\tilde{\Sigma}')$ satisfies axioms CM1 and CM2. 
Because $\tilde{\ex} = \mu_{x_l} \circ \ldots \circ \mu_{x_1}(\ex)$ and $\tilde{\ex}' = \mu_{f(x_l)} \circ \ldots \circ \mu_{f(x_1)}(\ex')$, every $(f,\tilde{\Sigma},\tilde{\Sigma}')$-biadmissible sequence $(y_1, \ldots, y_m)$ gives rise to a $(f,\Sigma,\Sigma')$-biadmissible sequence $(x_1, \ldots, x_l, y_1, \ldots, y_m)$. Let now $\tilde{y} \in \tilde{\XX}$ be such that $f(\tilde{y}) \in \tilde{\XX}'$. We have $\tilde{y} = \mu_{x_l} \circ \ldots \circ \mu_{x_1}(y)$ for a $y \in \XX$. If $f(y) \in \mathbb{Z}$, then with the same argument as above we have $\tilde{y} = y$ and thus $f(\tilde{y}) \in \mathbb{Z}$. Therefore whenever we have $f(\tilde{y}) \in \tilde{\XX}'$ we have $f(y) \in \XX'$ and by axiom CM3 for $f \colon \A(\Sigma) \to \A(\Sigma')$ we have
\begin{align*}
f(\mu_{y_m} \circ \ldots \circ \mu_{y_1}(\tilde{y})) &= f(\mu_{y_m} \circ \ldots \circ \mu_{y_1} \circ \mu_{x_l} \circ \ldots \circ \mu_{x_1}(y)) \\ &= \mu_{f(y_m)} \circ \ldots \circ \mu_{f(y_1)} \circ \mu_{f(x_l)} \circ \ldots \circ \mu_{f(x_1)}(f(y)) \\ &=\mu_{f(y_m)} \circ \ldots \circ \mu_{f(y_1)} (f(\tilde{y}))
\end{align*}
Thus axiom CM3 is satisfied for $f \colon \A(\tilde{\Sigma}) \to \A(\tilde{\Sigma}')$.
\end{proof}

We will show in Proposition \ref{P:almost injective} that a rooted cluster morphism is quite limited in its action on exchangeable variables of the initial seed: It has to be injective on the exchangeable variables that are not being sent to integers. Furthermore, it cannot map any coefficients to the same cluster variable to which it maps an exchangeable variable. It may however send two coefficients to the same cluster variable, as long as it is careful about their exchangeable neighbours.

\begin{definition}
Let $\Sigma =  (\XX,\ex,B=(b_{vw})_{v,w \in \XX})$ be a seed and let $x \in \XX$ be a cluster variable in $\Sigma$. We call a cluster variable $y \in \XX$ {\em a neighbour of $x$ in $\Sigma$}, if $b_{xy} \neq 0$. 
\end{definition}

\begin{remark}
Note that being neighbours is a symmetric relation: For a given seed  $\Sigma =  (\XX,\ex,B)$ a cluster variable $x \in \XX$ is a neighbour of $y \in \XX$ in $\Sigma$ if and only if $y$ is a neighbour of $x$ in $\Sigma$. We then say that {\em $x$ and $y$ are neighbours in $\Sigma$}.
\end{remark}

\begin{proposition}\label{P:almost injective}
Let $\Sigma = (\XX, \ex, B=(b_{vw})_{v,w \in \XX})$ and $\Sigma' = (\XX', \ex', B')$ be seeds and let $f \colon \A(\Sigma) \to \A(\Sigma')$ be a rooted cluster morphism. If $x \neq y$ are cluster variables of $\Sigma$ with $f(x) = f(y) \in \XX'$, then both $x$ and $y$ are coefficients of $\Sigma$. In that case for any $z \in \ex$ that is a neighbour of both $x$ and $y$ in $\Sigma$ and such that $f(z) \in \ex$, the entries $b_{zx}$ and $b_{zy}$ have the same sign.
\end{proposition}

\begin{proof}
Let $x \in \ex$ with $f(x) \in \XX'$. We want to show that $f(y) \neq f(x)$ for every cluster variable $y \in \XX \setminus x$. By axiom CM2 we have $f(x) \in \ex'$ and the sequence $(x)$ is $(f,\Sigma,\Sigma')$-biadmissible.  Let $y \in \XX$ with $y \neq x$. If $f(y) \in \mathbb{Z}$ then we have $f(x) \neq f(y)$, thus assume $f(y) \in \XX'$. By axiom CM3 we obtain
\[
f(y) = f(\mu_x(y)) = \mu_{f(x)}(f(y)).
\]
Assume for a contradiction that $f(y) = f(x) =:z' \in \ex'$. This would imply $z' = \mu_{z'}(z')$. Writing $B' = (b'_{xy})_{x,y \in \XX'}$ we thus would have
\[
(z')^2 = z' \mu_{z'}(z') = \prod\limits _{v \in \XX' \colon b'_{z'v}>0} v^{b'_{z'v}} + \prod\limits _{v \in \XX' \colon b'_{z'v}<0} v^{-b'_{z'v}},
\]
which contradicts algebraic independence of the cluster variables in $\XX'$.

We now prove the second part of the statement. Let $x \neq y \in \XX \setminus \ex$ be coefficients of $\Sigma$. Assume for a contradiction that $f(x) = f(y) = x' \in \XX'$ and there exists a $z \in \ex$ which is a neighbour of both $x$ and $y$ in $\Sigma$ with $f(z) \in \ex'$ such that $b_{zx}$ and $b_{zy}$ have opposite signs. Without loss of generality assume $b_{zx} >0$ and $b_{zy} < 0$. 
Then we have
\begin{align*}
f(z \mu_z(z)) &= \prod\limits _{v \in \XX \colon b_{zv}>0} f(v)^{b_{zv}} + \prod\limits _{v \in \XX \colon b_{zv}<0} f(v)^{-b_{zv}} \\
&=  x' \left(\prod\limits _{x \neq v \in \XX \colon b_{zv}>0} f(v)^{b_{zv}} + \prod\limits _{y \neq v \in \XX \colon b_{zv}<0} f(v)^{-b_{zv}}\right).
\end{align*}
By axiom CM3 this has to be equal to
\begin{align*}
f(z)\mu_{f(z)}(f(z)) & = \prod\limits _{v' \in \XX' \colon b'_{f(z)v'}>0} (v')^{b'_{f(z)v'}} + \prod\limits _{v' \in \XX' \colon b'_{f(z)v'}<0} (v')^{-b'_{f(z)v'}},
 \end{align*}
and since we either have $b'_{f(z)x'} \geq 0$ or $b'_{f(z)x'} < 0$ the cluster variable $x'$ cannot divide both summands on the right hand side.
This contradicts algebraic independence of the variables in $\XX'$.
\end{proof}

%

\subsection{Ideal rooted cluster morphisms} \label{S:ideal}
An ideal rooted cluster morphism is a rooted cluster morphism $f \colon \A(\Sigma) \to \A(\Sigma')$ whose image $f(\A(\Sigma))$ is the rooted cluster algebra with initial seed $f(\Sigma)$ the image of $\Sigma$. In the discussion before \cite[Problem~2.12]{ADS} (which asks for a characterization of all ideal rooted  cluster morphisms), the authors asked whether every rooted cluster morphism was ideal. In this section we answer the question by showing that not every rooted cluster morphism is ideal.

\begin{definition}[{\cite[Definition~2.8]{ADS}}]
Let \[\Sigma = (\XX, \ex, B) \quad \text{and} \quad \Sigma' = (\XX', \ex', B' = (b'_{vw})_{v,w\in \XX'})\] be seeds and let $f \colon \A(\Sigma) \to \A(\Sigma')$ be a rooted cluster morphism. Then the image $f(\Sigma)$ of the seed $\Sigma$ under the morphism $f$ is the seed \[f(\Sigma) = (f(\XX) \cap \XX' , f(\ex) \cap \ex', f(B) = (b'_{vw})_{v,w \in f(\XX)\cap \XX'}).\]
\end{definition}

Note that the exchangeable variables of the image seed $f(\Sigma)$ all are images of exchangeable variables of $\Sigma$. We might well have exchangeable variables of $\Sigma'$ that lie in the image $f(\XX \setminus \ex)$ of the coefficients of $\Sigma$ -- these are not exchangeable variables of $f(\Sigma)$.

If $f \colon \A(\Sigma) \to \A(\Sigma')$ is a rooted cluster morphism, then the seed $f(\Sigma)$ is an example of what is called a full subseed of the seed $\Sigma'$.

\begin{definition}[{\cite[Definition~4.9]{ADS}}]\label{D:full subseed}
Let \[\Sigma' = (\XX', \ex', B'=(b'_{vw})_{v,w \in \XX'})\] be a seed. A {\em full subseed} of $\Sigma'$ is a seed $\Sigma = (\XX,\ex,B=(b_{vw})_{v,w \in \XX})$ such that
\begin{itemize}
\item{$\XX \subseteq \XX'$,}
\item{$\ex \subseteq \ex'$,}
\item{$B$ is the submatrix of $B'$ formed by the entries labelled by $\XX \times \XX$, i.e.\;for all $v,w \in \XX$ we have $b_{vw} = b'_{vw}$.}
\end{itemize}
\end{definition}

\begin{remark}
Note that while all exchangeable variables of a full subseed of $\Sigma'$ have to be exchangeable variables of $\Sigma'$, cluster variables which are coefficients in the full subseed are not necessarily coefficients in $\Sigma'$.
\end{remark}

\begin{definition}[{\cite[Definition~2.11]{ADS}}]\label{D:ideal}
A rooted cluster morphism \[f \colon \A(\Sigma) \to \A(\Sigma')\] is called {\em ideal} if its image is the rooted cluster algebra with initial seed $f(\Sigma)$, i.e.\ if $f(\A(\Sigma)) = \A(f(\Sigma))$.
\end{definition}

In \cite[Lemma~2.10]{ADS} the authors showed that the inclusion $\A(f(\Sigma)) \subseteq f(\A(\Sigma))$ holds for any rooted cluster morphism $f \colon \A(\Sigma) \to \A(\Sigma')$. We can show that the converse is not true in general. 

\begin{theorem} \label{T:ideal}
Not every rooted cluster morphism is ideal.
\end{theorem}

\begin{proof}
We give an example of a rooted cluster morphism that is not ideal. Consider the seeds \[\Sigma = (\{a_1,a_2,x\}, \{x\}, 
\begin{tikzpicture}[baseline={([yshift={-\ht\strutbox}]1.north)}, scale=2.5,cap=round,>=latex]
        \tikzstyle{every node}=[font=\small]
\node[rectangle, draw] (1) at (0,0) {$a_1$};
\node (2) at (1,0) {$x$};
\node[rectangle, draw] (3) at (2,0) {$a_2$};

\draw[->] (1) -- (2);
\draw[->] (2) -- (3);
\end{tikzpicture}) \]
and
\[\Sigma' = (\{y_1,y_2\}, \{y_1,y_2\}, \begin{tikzpicture}[baseline={([yshift={-\ht\strutbox}]1.north)}, scale=2.5,cap=round,>=latex]
        \tikzstyle{every node}=[font=\small]
\node (1) at (0,0) {$y_1$};
\node (2) at (1,0) {$y_2$};

\draw[->] (1) -- (2);
\end{tikzpicture}). \]
As rings, the cluster algebras are isomorphic to 
\[\A(\Sigma) \cong \mathbb{Z}[a_1,a_2,x,x'] \big/ \langle xx' = a_1 + a_2 \rangle \]
and
\[\A(\Sigma') \cong \mathbb{Z}\bigg[y_1,y_2, \frac{1 + y_1}{y_2}, \frac{1+y_2}{y_1}, \frac{1+y_1+y_2}{y_1y_2}\bigg].\]
Consider the ring homomorphism $f \colon \A(\Sigma) \to \A(\Sigma')$ defined by the algebraic extension of the map which sends 
\begin{align*}
a_1 & \mapsto 1 & a_2 & \mapsto -1\\
x & \mapsto 0 & x' & \mapsto y_1.
\end{align*}
Because $f(xx') = 0 = f(a_1+ a_2)$ this is well-defined. Furthermore, it satisfies the axioms CM1 and CM2 for a rooted cluster morphism and because there are no $(f, \Sigma, \Sigma')$-biadmissible sequences it trivially satisfies axiom CM3 and thus is a rooted cluster morphism. The image of the seed $\Sigma$ is $f(\Sigma) = (\emptyset, \emptyset, \emptyset)$ and thus as a ring we have $\A(f(\Sigma)) \cong \mathbb{Z}$. However, the image of the cluster algebra $\A(\Sigma)$ is $f(\A(\Sigma)) \cong \mathbb{Z}[y_1]$.
\end{proof}

\subsection{The category of rooted cluster algebras}
Considering rooted cluster algebras and rooted cluster morphisms gives rise to a category.

\begin{definition}[{\cite[Definition~2.6]{ADS}}]
The {\em category of rooted cluster algebras $Clus$} is the category which has as objects rooted cluster algebras and as morphisms rooted cluster morphisms.
\end{definition}

In \cite[Section~2]{ADS} it was shown that $Clus$ satisfies the axioms of a category. In particular, axiom CM2 for rooted cluster morphisms is necessary to ensure that compositions of rooted cluster morphisms are again rooted cluster morphisms, as the following example illustrates.

\begin{example}
Consider the seeds 
\begin{align*}
\Sigma_1 &= \left( \{x_1,x_2,x_3\}, \{x_2\},\begin{tikzpicture}[scale=2.5,cap=round,>=latex, baseline={([yshift={-\ht\strutbox}]x_1.north)}]
        \tikzstyle{every node}=[font=\small]) 
        \node[rectangle, draw] (x_1) at (0,0) {$x_1$};
        \node (x_2) at (1,0) {$x_2$};
         \node[rectangle, draw] (x_3) at (2,0) {$x_3$};
         \draw[->] (x_1) -- (x_2);
         \draw[->] (x_2) -- (x_3);
         \end{tikzpicture} \right), \\
\Sigma_2 &= \left( \{z\}, \emptyset,\begin{tikzpicture}[scale=2.5,cap=round,>=latex, baseline={([yshift={-\ht\strutbox}]x_1.north)}]
        \tikzstyle{every node}=[font=\small]) \node[rectangle, draw] at (0,0) {$z$}; \end{tikzpicture} \right) \\ 
\Sigma_3 &= (\{y_1,y_2\}, \{y_1,y_2\}, \xymatrix{y_1 \ar[r] & y_2})
\end{align*} 
with associated cluster algebras \[\A(\Sigma_1) = \mathbb{Z}[x_1,x_2,x_3, \frac{x_1+x_3}{x_2}], \quad A(\Sigma_2) = \ZZ[z]\] and \[\A(\Sigma_3) = \ZZ[y_1,y_2,\frac{1+y_2}{y_1}, \frac{1+y_1}{y_2}, \frac{1+y_1+y_2}{y_1y_2}].\] Consider the ring homomorphisms $f \colon \A(\Sigma_1) \to \A(\Sigma_2)$, which is defined by sending $x_i \mapsto z$ for all $i = 1,2,3$, and $g \colon \A(\Sigma_2) \to \A(\Sigma_3)$ defined by sending $z \mapsto y_1$. Both $f \colon \A(\Sigma_1) \to \A(\Sigma_2)$ and $g \colon \A(\Sigma_2) \to \A(\Sigma_3)$ satisfy axiom CM1, but $f$ does not satisfy axiom CM2. Axiom CM3 is satisfied trivially by both $f$ and $g$, since there are neither $(f, \Sigma_1, \Sigma_2)$ nor $(g, \Sigma_2, \Sigma_3)$-biadmissible sequences. However, the composition $g \circ f$ does not satisfy axiom CM3: Consider the $(g \circ f, \Sigma_1, \Sigma_3)$-biadmissible sequence $(x_2)$. We have
\[
g \circ f (\mu_{x_2}(x_2)) = g \circ f \left(\frac{x_1 + x_3}{x_2}\right) = g(2) = 2
\]
but
\[
\mu_{g \circ f(x_2)}(g \circ f(x_2)) = \mu_{y_1}(y_1) = \frac{1 + y_2}{y_1}.
\]
\end{example}

\subsection{Coproducts and connectedness of seeds} \label{ss:Connectedness of seeds and coproducts}

Assem, Dupont and Schiffler showed in \cite[Lemma~5.1]{ADS} that countable coproducts exist in the category $Clus$ of rooted cluster algebras. Taking coproducts of a countable family $\{\A(\Sigma_i)\}_{i \in I}$ of rooted cluster algebras amounts to taking what can be intuitively described as the disjoint union $\Sigma$ of their seeds. The seeds $\Sigma_i$ will be full subseeds of the seed $\Sigma$ which are mutually disconnected.

\begin{definition}\label{D:connected}
Let $\Sigma = (\XX,\ex,B)$ be a seed. A sequence $x_0, x_1, \ldots, x_{l}$ of cluster variables in $\XX$ with $l  \geq 0$ such that for any $0 \leq i < l$ the cluster variables $x_i$ and $x_{i+1}$ are neighbours in $\Sigma$ is called a {\em path of length $l$ in $\Sigma$.} We call two cluster variables $x,y \in \XX$ {\em connected in $\Sigma$}, if there exists a path $x_0, \ldots, x_l$ of finite length $l \geq 0$ in $\Sigma$ such that $x = x_0$ and $y = x_l$. We call the seed $\Sigma$ {\em connected} if any two cluster variables $x,y \in \XX$ are connected in $\Sigma$. We call a rooted cluster algebra $\A(\Sigma)$ {\em connected} if its initial seed $\Sigma$ is connected.
\end{definition}


We can decompose any seed into its connected components. For a seed \[\Sigma = (\XX, \ex, B = (b_{vw})_{v,w \in \XX})\] and an element $x \in \XX$ of its cluster, the {\em connected component of $x$ in $\Sigma$} is the full connected subseed $\Sigma_x$ of $\Sigma$ consisting of those cluster variables in $\XX$ that are connected to $x$ and such that all coefficients in $\Sigma_x$ are also coefficients in $\Sigma$, i.e.\ 
\[
\Sigma_x = (\XX_x, \ex \cap \XX_x, B_x = (b_{vw})_{v,w \in \XX_x}),
\]
where $\XX_x = \{y \in \XX \mid x \text{ and }y \text{ are connected in } \Sigma\}$.
The decomposition is given as follows: Let $\{\Sigma_j = (\XX_j, \ex_j, B_j = (b^j_{vw})_{v,w \in \XX_j})\}_{j \in I}$ for a countable index set $I$ be the set of mutually distinct connected components in $\Sigma$. Since no vertex in $\XX_i$ is connected to any vertex in $\XX_j$ for $i \neq j \in I$, we have $b_{xy} = 0$ for $x \in \XX_i$ and $y \in \XX_j$ with $i \neq j \in I$. Thus we have $\XX = \bigcup_{j \in I} \XX_j$ and $\ex = \bigcup_{j \in I} \ex_j$, and the matrix $B$ has the matrices $B_j$ for $j \in I$ as block-diagonal entries, i.e.\;$b_{vw} = b^j_{vw}$ if $v,w \in \XX_j$ for some $j \in I$ and $b_{vw} = 0$ otherwise. 

\begin{remark}
It is a direct consequence of the definition of matrix mutation (cf.\ Definition \ref{D:mutation of a seed}) that mutation of seeds respects connected components.
\end{remark}

Conversely we can build a new seed from a countable collection of seeds by taking the disjoint union of the clusters and the exchangeable variables and constructing a big matrix which contains all of their exchange matrices as block-diagonal entries:  Let $\{\Sigma_j = (\XX_j, \ex_j, B_j)\}_{ j \in I}$ be a countable collection of seeds. Denote by $\bigsqcup$ the disjoint union and set \[\bigsqcup_{j \in I} \Sigma_j := (\bigsqcup_{j \in I} \XX_j, \bigsqcup_{j \in I} \ex_j, B), \] where $B$ is the block-diagonal matrix with blocks $B_j$ for $j \in I$. The analogous construction for rooted  cluster algebras is taking coproducts; by \cite[Lemma~5.1]{ADS}, the category $Clus$ of rooted cluster algebras admits countable coproducts $\coprod$ and for a countable index set $I$ we have \[\coprod_{j \in I}\A(\Sigma_j) \cong \A(\bigsqcup_{j \in I} \Sigma_j).\]
The seeds $\Sigma_j$ for $j \in I$ are mutually disconnected full subseeds of $\Sigma$. On the other hand, since we can decompose any given seed into its connected components and there are only countably many cluster variables, given a rooted cluster algebra $\A(\Sigma)$ we can write it as a countable coproduct of connected rooted cluster algebras.

\begin{remark}\label{R:connected implies countable}
If we omit the countability assumption for clusters of seeds (cf.\ Remark \ref{R:countability}), then uncountable coproducts exist: This follows directly from the proof of \cite[Lemma~5.1]{ADS}, where the countability assumption is solely needed for the cluster of the coproduct to be countable. All of the other arguments go through directly.

In fact, having uncountably many connected components is the only way for a seed to have an uncountable cluster; if a seed is connected, then the fact that it has a countable cluster is automatic: Let $\Sigma = (\XX, \ex, B)$ be a connected seed and let $x \in \XX$. For $l \geq 0$ we set \[\XX_l = \{y \in \XX \mid \text{ $x$ and $y$ are connected by a path of length $l$ in $\Sigma$}\}.\] Because $B$ is locally finite, for each $l  \geq 0$, the set $\XX_l$ is finite and because $\Sigma$ is connected, we have $\XX = \bigcup_{l \geq 0} \XX_l$. Therefore, $\XX$ is countable.

Thus every connected component of a seed has a countable cluster.  As a consequence, one does not currently gain much from considering uncountable clusters. The usual operations on seeds, namely mutations along (finite) admissible sequences, affect only finitely many connected components and hence only operate on a countable full subseed which is not connected to its invariant complement. So for all practical purposes one can restrict to working with countable seeds without any substantial loss of generality.
\end{remark}

Let us consider again the example of a countable triangulation $\T$ of the closed disc with marked points $\Z \subseteq S^1$. Note that by definition of the seed $\Sigma_\T$ associated to $\T$ two cluster variables $\alpha, \beta \in \T$ are neighbours in $\Sigma_\T$ if and only if the arcs $\alpha$ and $\beta$ are sides of a common triangle in $\T$ and they are connected in $\Sigma_\T$ if and only if there exists a $k \geq 0$ and a sequence of arcs $\gamma_0, \ldots, \gamma_k$, such that $\alpha = \gamma_0$ and $\beta = \gamma_k$ and for all $0 \leq i \leq k$ the arcs $\gamma_i$ and $\gamma_{i+1}$ are sides of a common triangle in $\T$. It turns out that the connected components of $\Sigma_\T$ depend on the behaviour of arcs in $\T$ in the neighbourhood of limit points of $\Z$.

\begin{definition}\label{D:fountain}
Let $\Z \subseteq S^1$. We say that a sequence $\{z_i\}_{i \in \mathbb{Z}_{\geq 0}}$ of points in $\Z$ converging to $z$ converges to $z \in S^1$ {\em from the right}, if for any $x \in S^1$ the set $[x,z) \cap \Z$ is infinite and the set $(z,x] \cap \Z$ is finite. We say that it converges to $z \in S^1$ {\em from the left}, if for any $x \in S^1$ the set $[x,z) \cap \Z$ is finite and the set $(z,x] \cap \Z$ is infinite. We say that it converges to $z \in S^1$ {\em from both sides}, if for any $x \in S^1$ both the set $[x,z) \cap \Z$ and the set $(z,x] \cap \Z$ are infinite.

Let $\{a_i,b_i\}_{i \in \mathbb{Z}_{\geq 0}}$ be a sequence of arcs of $\Z$ and let the sequence of endpoints converge to $a = \varinjlim a_i \in S^1$  and $b = \varinjlim b_i \in S^1$. If both sequences of endpoints $\{a_i\}_{i \in \mathbb{Z}_{\geq 0}}$ and $\{b_i\}_{i \in \mathbb{Z}_{\geq 0}}$ are non-constant, we say that the sequence $\{a_i,b_i\}_{i \in \mathbb{Z}_{\geq 0}}$ of arcs is a {\em nest} if $a = b$ and we say that it is a {\em half-nest} if $a \neq b$.

If the sequence $\{a_i\}_{i \in \mathbb{Z}_{\geq 0}}$ is constant and the sequence $\{b_i\}_{i \in \mathbb{Z}_{\geq 0}}$ is non-constant, we say that the sequence  $\{a_i,b_i\}_{i \in \mathbb{Z}_{\geq 0}}$ of arcs is a {\em right-fountain at $a$ converging to $b$}, if $\{b_i\}_{i \in \mathbb{Z}_{\geq 0}}$ converges to $b$ from the right, we say that it is a {\em left-fountain  at $a$ converging to $b$}, if $\{b_i\}_{i \in \mathbb{Z}_{\geq 0}}$ converges to $b$ from the left and we say that it is a {\em fountain  at $a$ converging to $b$}, if $\{b_i\}_{i \in \mathbb{Z}_{\geq 0}}$ converges to $b$ from both sides. We call a sequence $\{a_i, b_i\}_{i \in \ZZ_{\geq 0}}$ of arcs in $\Z$ a {\em split fountain converging to $b$}, if it can be partitioned into a left fountain $\{a_l, b_{i}\}_{i \in \ZZ_{< 0}}$ at $a_l \in \Z$ converging to $b \in S^1$ and a right fountain $\{a_r, b_i\}_{i \in \ZZ_{\geq 0}}$ at $a_r \in \Z$ converging to $b$ with $a_l \neq a_r$.
\end{definition}

To determine the connected components of the seed $\Sigma_\T$ associated to a given countable triangulation $\T$ of the closed disc with marked points $\Z$ it is helpful to view any half-nest, fountain and right-or left-fountain $\{a_i,b_i\}_{i \in \mathbb{Z}_{\geq 0}}$ as converging  to an arc $\{a,b\}$ of the topological closure $\overline{\Z}$ of $\Z \subseteq S^1$, where  $a = \varinjlim a_i$  and $b = \varinjlim b_i$.  Let $\T$ be a triangulation of the closed disc with marked points $\Z$ and let $\{a,b\}$ be an arc of $\overline{\Z}$ such that there is a half-nest, fountain and right-or left-fountain in $\T$ converging $\{a,b\}$. Then we call $\{a,b\}$ a {\em limit arc of $\T$}. Figure \ref{fig:fountains} provides an illustration of a left-fountain, a right-fountain and a fountain, while Figure \ref{fig:nests} illustrates a half-nest and a nest.

\begin{figure}
\centering{
\begin{tikzpicture}[scale = 2, font = \footnotesize, font = \sffamily, font=\sansmath\sffamily, cap=round,>=latex]
        \tikzstyle{every node}=[font=\small]
	\node at (60:0.6cm){$a$};
	\node at (220:0.6) {$b$};
	\node at (220:0.5){$\bullet$};

	\draw[black!20] (220:0.5) -- (60:0.5);
	
	\draw[] (60:0.5) -- (210:0.5);
	\draw[] (60:0.5) -- (200:0.5);
	\draw[] (60:0.5) -- (180:0.5);
	\draw[] (60:0.5) -- (195:0.5);
	\draw[] (60:0.5) -- (205:0.5);
	\draw[] (60:0.5) -- (210:0.5);
	\draw[] (60:0.5) -- (150:0.5);

	\draw (60:0.5cm) -- (180:0.5);

	\draw[dotted] (215:0.4) -- (205:0.4);
        	\draw (0,0) circle(0.5cm);
  \end{tikzpicture}
\begin{tikzpicture}[scale = 2, font = \footnotesize, font = \sffamily, font=\sansmath\sffamily, cap=round,>=latex]
        \tikzstyle{every node}=[font=\small]
	\node at (60:0.6cm){$a$};
	\node at (220:0.6) {$b$};
	\node at (220:0.5){$\bullet$};
	
	\draw[black!20] (220:0.5) -- (60:0.5);

	\draw[] (60:0.5) -- (230:0.5);
	\draw[] (60:0.5) -- (235:0.5);
	\draw[] (60:0.5) -- (245:0.5);
	\draw[] (60:0.5) -- (250:0.5);
	\draw[] (60:0.5) -- (260:0.5);

	\draw[dotted] (235:0.4) -- (220:0.4);
        	\draw (0,0) circle(0.5cm);
  \end{tikzpicture}
\begin{tikzpicture}[scale = 2, font = \footnotesize, font = \sffamily, font=\sansmath\sffamily, cap=round,>=latex]
        \tikzstyle{every node}=[font=\small]
	\node at (60:0.6cm){$a$};
	\node at (220:0.6) {$b$};
	\node at (220:0.5){$\bullet$};
	
	\draw[] (60:0.5) -- (210:0.5);
	\draw[] (60:0.5) -- (200:0.5);
	\draw[] (60:0.5) -- (180:0.5);
	\draw[] (60:0.5) -- (195:0.5);
	\draw[] (60:0.5) -- (205:0.5);
	\draw[] (60:0.5) -- (210:0.5);
	\draw[] (60:0.5) -- (150:0.5);

	\draw[black!20] (220:0.5) -- (60:0.5);
	
	\draw (60:0.5cm) -- (180:0.5);

	\draw[] (60:0.5) -- (230:0.5);
	\draw[] (60:0.5) -- (235:0.5);
	\draw[] (60:0.5) -- (245:0.5);
	\draw[] (60:0.5) -- (250:0.5);
	\draw[] (60:0.5) -- (260:0.5);

	\draw[dotted] (235:0.4) -- (220:0.4);
	\draw[dotted] (215:0.4) -- (205:0.4);
        	\draw (0,0) circle(0.5cm);
  \end{tikzpicture}
}
\caption{Triangulations of the closed disc consisting of (from left to right) a right fountain at $a \in \Z$, a left fountain at $a$ and a fountain at $a$, all converging to the limit arc $\{a,b\}$ of the respective triangulation}\label{fig:fountains}
\end{figure}
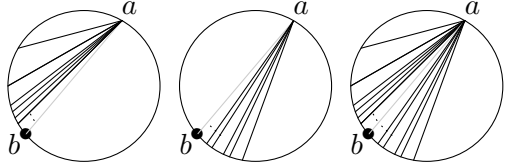

\begin{figure}
\centering{
\begin{tikzpicture}[scale = 2, font = \footnotesize, font = \sffamily, font=\sansmath\sffamily, cap=round,>=latex]
        \tikzstyle{every node}=[font=\small]
	\node at (60:0.6cm){$a$};
	\node at (200:0.6) {$b$};
	\node at (60:0.5){$\bullet$};
	\node at (200:0.5){$\bullet$};
	
	\draw[] (55:0.5) -- (225:0.5);
	\draw[] (225:0.5) -- (45:0.5);
	\draw[] (45:0.5) -- (230:0.5);
	\draw[] (30:0.5) -- (240:0.5);
	\draw[] (240:0.5) -- (25:0.5);
	\draw[] (240:0.5) -- (20:0.5);
	\draw[] (240:0.5) -- (10:0.5);
	\draw[] (10:0.5) -- (260:0.5);
	\draw[] (55:0.5) -- (210:0.5);

	\draw[black!20] (200:0.5) -- (60:0.5);
	\draw[dotted] (130:0.1) -- (130:0.2);

        	\draw (0,0) circle(0.5cm);
  \end{tikzpicture}
\begin{tikzpicture}[scale = 2, font = \footnotesize, font = \sffamily, font=\sansmath\sffamily, cap=round,>=latex]
\tikzstyle{every node}=[font=\small]
        
	\node at (60:0.6cm){$a$};
	\node at (60:0.5){$\bullet$};
	\draw[dotted] (65:0.4) -- (65:0.5);
	
	\draw[] (98:0.5) -- (30:0.5);
	\draw[] (20:0.5) -- (110:0.5);
	\draw[] (110:0.5) -- (15:0.5);
	\draw[] (15:0.5) -- (120:0.5);
	\draw[] (120:0.5) -- (10:0.5);
	\draw[] (10:0.5) -- (130:0.5);
	\draw[] (10:0.5) -- (135:0.5);
	\draw[] (10:0.5) -- (140:0.5);
	\draw[] (140:0.5) -- (5:0.5);

        	\draw (0,0) circle(0.5cm);
  \end{tikzpicture}
}
\caption{A triangulation of the closed disc consisting of a half nest converging to the limit arc $\{a,b\}$ of the triangulation, and a nest where the sequence of endpoints converges to $a$} \label{fig:nests}
\end{figure}
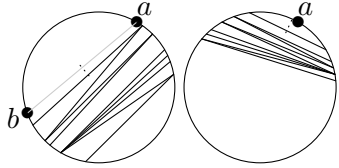

\begin{lemma}\label{L:connected components of triangulation}
Let $\T$ be a countable triangulation of the closed disc with marked points $\Z \subseteq S^1$. Two arcs $\{x_0,x_1\} \neq \{y_0,y_1\}$ are connected in $\Sigma_\T$ if and only if there is no limit arc $\{a,b\}$ of $\T$
such that $x_0,x_1 \in [a,b]$ and $y_0,y_1 \in [b,a]$ or vice-versa.
\end{lemma}

\begin{proof}
First assume that there exists
a limit arc $\{a,b\}$ of $\T$. It is straightforward to check that there cannot be a finite sequence of arcs connecting any arc with endpoints in $[a,b]$ with a distinct arc with endpoints in $[b,a]$, since there are infinitely many arcs from the right-or left-fountain, fountain or half-nest in $\T$ converging to the limit arc $\{a,b\}$ of $\T$ in between. 

On the other hand, assume that $\{x_0,x_1\}$ and  $\{y_0,y_1\}$ are not connected. Without loss of generality let $(x_0,x_1) \subseteq (x_0,y_0) \subseteq (x_0,y_1)$. We can construct a sequence of arcs $\{a_i,b_i\}_{i \in \mathbb{Z}_{\geq 0}}$ by setting $a_0 = x_0$ and $b_0 = x_1$ and for $i \geq 1$ choosing $\{a_i,b_i\}$ such that $\{a_{i-1},b_{i-1}\}$ and $\{a_i,b_i\}$ are sides of a common triangle in $\T$ and such that $a_i \in [y_1, a_{i-1}]$ and $b_i \in [b_{i-1},y_0]$. Because $\{x_0,x_1\}$ and $\{y_0,y_1\}$ are not connected, $a_i$ and $b_i$ are well-defined for all $i \geq 0$ and at least one of the sequences $\{a_i\}_{i \in \ZZ}$ and $\{b_i\}_{i \in \ZZ}$ is not constant.  Both sequences of endpoints $\{a_i\}_{i \in \mathbb{Z}_{\geq 0}}$ and $\{b_i\}_{i \in \mathbb{Z}_{\geq 0}}$ are monotone and bounded above and below and thus $\{a_i,b_i\}_{i \in \mathbb{Z}_{\geq 0}}$ is a half-nest, fountain or right-or left-fountain converging to a limit arc $\{a,b\}$ of $\T$
such that $x_0,x_1 \in [a,b]$ and $y_0,y_1 \in [b,a]$.
\end{proof}

\begin{remark}\label{R:connected components of triangulation}
For a given countable triangulation $\T$ with marked points $\Z \subseteq S^1$ the limit arcs $\{a,b\}$ of $\T$ partition the seed $\Sigma_\T$ associated to $\T$ into connected components.  If a limit arc $\{a,b\}$ of $\T$ is not an arc of $\Z$ (in particular this is always the case if $\Z$ is discrete) then it divides $\Sigma_\T$ into two mutually disconnected components. If the limit arc $\{a,b\}$ of $\T$ is an arc of $\Z$, then it is an arc in $\T$ (because it cannot cross any arc of $\T$) and it provides an additional connected component, consisting only of the arc $\{a,b\}$ itself. Figure \ref{fig:partition} provides an illustration of the partition of a triangulation into connected components.
\end{remark}

\begin{figure}
\centering{
\begin{tikzpicture}[scale = 2, font = \footnotesize, font = \sffamily, font=\sansmath\sffamily,cap=round,>=latex]
        \tikzstyle{every node}=[font=\small]
	\node at (220:0.5){$\bullet$};
	\node at (130:0.5){$\bullet$};
	\node at (35:0.5){$\bullet$};
	\node at (340:0.5){$\bullet$};

	\draw[black!20] (60:0.5) -- (120:0.5);
	\draw[black!20] (60:0.5) -- (110:0.5);
	\draw[black!20] (60:0.5) -- (105:0.5);
	\draw[black!20] (60:0.5) -- (100:0.5);
	\draw[black!20] (60:0.5) -- (90:0.5);

	\draw[black!20] (60:0.5) -- (230:0.5);
	\draw[black!20] (60:0.5) -- (235:0.5);
	\draw[black!20] (60:0.5) -- (245:0.5);
	\draw[black!20] (60:0.5) -- (250:0.5);
	\draw[black!20] (60:0.5) -- (260:0.5);

	\draw[black!20] (60:0.5) -- (215:0.5);
	\draw[black!20] (60:0.5) -- (135:0.5);
	\draw[black!20] (60:0.5) -- (150:0.5);
	\draw[black!20] (60:0.5) -- (160:0.5);
	\draw[black!20] (60:0.5) -- (180:0.5);
	\draw[black!20] (60:0.5) -- (185:0.5);
	\draw[black!20] (60:0.5) -- (190:0.5);
	\draw[black!20] (60:0.5) -- (200:0.5);
	\draw[black!20] (60:0.5) -- (210:0.5);

	\draw (220:0.5) -- (60:0.5);
	\draw (130:0.5) -- (60:0.5);

	\draw[black!20] (270:0.5) -- (60:0.5);
	\draw[black!20] (270:0.5) -- (50:0.5);
	\draw[black!20] (270:0.5) -- (45:0.5);
	\draw[black!20] (270:0.5) -- (40:0.5);
	\draw (270:0.5) -- (35:0.5);
	\draw[black!20] (270:0.5) -- (30:0.5);
	\draw[black!20] (270:0.5) -- (20:0.5);

	\draw[black!20] (270:0.5) -- (15:0.5);
	\draw[black!20] (280:0.5) -- (15:0.5);
	\draw[black!20] (280:0.5) -- (5:0.5);
	\draw[black!20] (300:0.5) -- (5:0.5);
	\draw[black!20] (320:0.5) -- (0:0.5);
	\draw[black!20] (320:0.5) -- (350:0.5);
	\draw[black!20] (330:0.5) -- (350:0.5);

        	\draw[black!20] (0,0) circle(0.5cm);
  \end{tikzpicture}
  }
\caption{Partition of a triangulation $\T$ of the closed disc with marked points $\Z \subseteq S^1$ into connected components; arcs in $\T$ are drawn in grey, limit arcs of $\T$ are drawn in black and limit points of $\Z$ are marked by bullets} \label{fig:partition}
\end{figure}
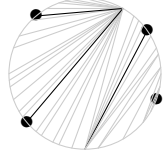

\begin{lemma}\label{L:connected triangulation}
Let $\T$ be a countable triangulation of the closed disc with marked points $\Z \subseteq S^1$ and let $\Sigma_\T$ be the associated seed. Then the rooted cluster algebra $\A(\Sigma_\T)$ is isomorphic to a coproduct of connected rooted cluster algebras associated to countable triangulations of the closed disc.
\end{lemma}

\begin{proof}
The limit arcs of $\T$ partition the seed $\Sigma_\T$ associated to $\T$ into countably many (since $\T$ is countable) connected components $\{\Sigma_{\T_i}\}_{i \in I}$ for some countable index set $I$. 
Thus by the discussion in Section \ref{ss:Connectedness of seeds and coproducts} the rooted cluster algebra $\A(\Sigma_{\T})$ is isomorphic to the countable coproduct $\coprod_{i \in I}\A(\Sigma_{\T_i})$.
\end{proof}

\subsection{Isomorphisms of rooted cluster algebras} 
An isomorphism of rooted cluster algebras implies a close combinatorial relation between their initial seeds. First, we introduce some useful terminology.

Let $\Sigma = (\XX, \ex, B)$ be a seed. We say that two cluster variables $x,y \in \XX$ are {\em connected by exchangeable variables in $\Sigma$} if there exists a path $x_0, x_1, \ldots, x_{l}$ of finite length $l \geq 0$ in $\Sigma$ (see Definition \ref{D:connected}), such that $x = x_0$, $y = x_{l}$, and $x_1, \ldots, x_{l-1}$ lie in $\ex$. Further, if $l \in \{0,1\}$ then at least one of $x_0$ and $x_l$ has to lie in $\ex$. Thus, two coefficients that are neighbours are not necessarily connected by exchangeable variables and a coefficient is not necessarily connected to itself by exchangeable variables. For an exchangeable variable $x \in \ex$ we define the {\em exchangeably connected component of $x$ in $\Sigma$} to be the full subseed $\Sigma^{ex}_x = (\XX^{ex}_x, \ex \cap \XX^{ex}_x, B_x = (b_{vw})_{v,w \in \XX^{ex}_x})$ of $\Sigma$ where
\[
\XX^{ex}_x = \{y \in \XX \mid x \text{ and $y$ are connected by exchangeable variables in $\Sigma$}\}.
\]

Partitioning a seed into its exchangeably connected components can be useful when studying mutations of a seed; mutations within an exchangeably connected component leave all other exchangeably connected components unchanged.

\begin{remark}\label{R:mutation of exchangeably connected components}
We can decompose any seed $\Sigma = (\XX, \ex, B)$ into its exchangeably connected components $\{\Sigma_j = (\XX_j, \ex_j, B_j = (b^j_{vw})_{v,w \in \XX_j})\}_{j \in I}$, where $I$ is a countable index set and $\ex = \sqcup_{j \in I} \ex_j$. Mutation along a $\Sigma_j$-admissible sequence leaves all other exchangeably connected components unchanged (up to entries of the exchange matrix labelled by coefficients), i.e. if $(x_1, \ldots, x_l)$ is a $\Sigma_i$-admissible sequence and $y \in \ex_j$ with $i \neq j \in I$ we have
\[(\mu_{x_l} \circ \ldots \circ \mu_{x_1}(\Sigma))_y = (\XX_j, \ex_j, \tilde{B}_j = (\tilde{b}^j_{vw})_{v,w \in \XX_j}),\]
where $\tilde{b}^j_{vw} = b^j_{vw}$ for all $v,w \in \ex^j$.
This follows directly from the fact that mutation at an exchangeable variable $x$ of $\Sigma$ only affects entries in the exchange matrix that are labelled by neighbours of $x$ in $\Sigma$. No entries that are labelled by exchangeable variables of any other exchangeably connected component are affected. 

In particular, the same holds true for connected components rather than just exchangeably connected components: Mutation in one connected component does not affect any other connected component, since two cluster variables in different connected components are necessarily in different exchangeably connected components. 
\end{remark}

\begin{definition}
We call two seeds $\Sigma = (\XX, \ex, B = (b_{vw})_{v,w \in \XX})$ and $\Sigma' = (\XX',\ex',B'=(b'_{vw})_{v,w \in \XX'})$ similar, if there exists a bijection $\varphi \colon \XX \to \XX'$ restricting to a bijection $\varphi \colon \ex \to \ex'$ such that for every exchangeable variable $x \in \XX$ the exchangeably connected component $\Sigma^{ex}_x$ of $x$ in $\Sigma$ is isomorphic (cf.\ Definition \ref{D:seed}) to the exchangeably connected component $\Sigma^{ex}_{\varphi(x)}$ of $\varphi(x)$ in $\Sigma'$ or to its opposite seed $(\Sigma^{ex}_{\varphi(x)})^{op}$. 
\end{definition}

\begin{theorem}\label{T:similar}
The rooted cluster algebras $\A(\Sigma)$ and $\A(\Sigma')$ are isomorphic if and only if  their initial seeds $\Sigma$ and $\Sigma'$ are similar.
\end{theorem}

This statement can be derived from \cite[Section~3]{ADS}. However, we consider the case where a seed might consist of several exchangeably connected components, so for the convenience of the reader we give a short proof.

\begin{proof}
Let $\Sigma = (\XX, \ex, B = (b_{vw})_{v,w \in \XX})$ and $\Sigma' = (\XX', \ex', B' = (b'_{vw})_{v,w \in \XX'})$ be similar via a bijection $\varphi \colon \XX \to \XX'$. It involves some calculations to check that $\varphi$ induces a rooted cluster morphism; in the interest of not giving a rather technical argument twice, we refer to a result that will be proved in Section \ref{S:Behaviour of rooted cluster morphisms without specializations}: In Theorem \ref{T:rooted cluster morphisms without specializations} we give three necessary and sufficient conditions for a map between clusters of seeds to give rise to a rooted cluster morphism. It is straightforward to check that $\varphi$ satisfies all of these: Since it is a bijection restricting to a bijection from $\ex$ to $\ex'$ it satisfies conditions (1) and (2). It satisfies condition (3) because for every two exchangeable cluster variables $x$ and $y$ in the same exchangeably connected component of $\Sigma$ we have $b'_{\varphi(x)w'} = b_{x\varphi^{-1}(w')}$ and $b'_{\varphi(y)w'} = b_{y\varphi^{-1}(w')}$ or $b'_{\varphi(x)w'} = -b_{x\varphi^{-1}(w')}$ and $b'_{\varphi(y)w'} = -b_{y\varphi^{-1}(w')}$  for all $w' \in \XX'$ . Thus it induces a rooted cluster morphism $f \colon \A(\Sigma) \to \A(\Sigma')$. For the same reasons, the inverse $\varphi^{-1} \colon \XX' \to \XX$ induces a rooted cluster morphism $g \colon \A(\Sigma') \to \A(\Sigma)$. It remains to check that $f$ and $g$ are mutual inverses as rooted cluster morphisms.

Let $x$ be a cluster variable in $\A(\Sigma)$. It is of the form $x = \mu_{x_l} \circ \ldots \circ \mu_{x_1}(y)$ for some $y \in \XX$ and a $\Sigma$-admissible sequence $(x_1, \ldots, x_l)$. By induction on the length $l$ of the admissible sequence, we show that $g \circ f(x) = x$ and thus $g \circ f$ is the identity on $\A(\Sigma)$: If $l = 0$ we have $g \circ f (x) = \varphi^{-1} \circ \varphi(x) = x$. If $g \circ f$ is the identity on all cluster variables which can be written as a mutation along a $\Sigma$-admissible sequence of length $l-1$, then in particular $g \circ f (\mu_{x_{l-1} \circ \ldots \circ \mu_{x_1}}(y)) = \mu_{x_{l-1}} \circ \ldots \circ \mu_{x_1} (y)$ and $g \circ f (x_{l}) = x_l$ and $(x_1, \ldots, x_l)$ is $(g\circ f, \Sigma, \Sigma)$-biadmissible. Thus by axiom CM3 for $g \circ f$ we have 
\begin{align*}
g \circ f(x) &= \mu_{g \circ f(x _{l})} \circ \ldots \circ \mu_{g \circ f(x _{1})}(g \circ f (y)) \\
&= \mu_{x_l} \circ \mu_{g \circ f(x _{l-1})} \circ \ldots \circ \mu_{g \circ f(x_1)}(g \circ f (y)) \\
&= \mu_{x_l}(g \circ f (\mu_{x_{l}} \circ \ldots \circ \mu_{x_1}(y))) \\
&= \mu_{x_l} \circ \ldots \circ \mu_{x_1} (y) = x.
\end{align*}
The argument for $f \circ g$ being the identity on $\A(\Sigma')$ is symmetric and thus $f$ is a rooted cluster isomorphism.

On the other hand, if $f \colon \A(\Sigma) \to \A(\Sigma')$ is a rooted cluster isomorphism then by \cite[Corollary~3.2]{ADS} it induces a bijection $\varphi \colon \XX \to \XX'$. By condition (3) of Theorem \ref{T:rooted cluster morphisms without specializations} it follows that $\Sigma$ and $\Sigma'$ are similar under this bijection.
\end{proof}

\subsection{Rooted cluster morphisms without specializations} \label{S:Behaviour of rooted cluster morphisms without specializations}

The definition of rooted cluster morphisms (see Definition \ref{D:rooted cluster morphism}) allows cluster variables to be sent to integers. Sending a cluster variable to an integer is called a {\em specialization}. Given the seeds $\Sigma = (\XX, \ex, B)$ and $\Sigma' = (\XX', \ex', B')$ we call a rooted cluster morphism $f \colon \A(\Sigma) \to \A(\Sigma')$ a {\em rooted cluster morphism without specializations}, if $f(\XX) \subseteq \XX'$, i.e.\ if all cluster variables get sent to cluster variables. 

\begin{lemma}\label{L:biadmissible}
Let $f \colon \A(\Sigma) \to \A(\Sigma')$ be a rooted cluster morphism without specializations. Then every $\Sigma$-ad\-missi\-ble sequence is $(f,\Sigma,\Sigma')$-biadmissible.
\end{lemma}

\begin{proof}
We prove the claim by induction on the length $l$ of a $\Sigma$-admissible sequence. It is satisfied trivially for sequences of length $l=0$ . Assume now that it is satisfied for all $\Sigma$-admissible sequences of lengths at most $l \geq 0$ and let $(x_1, \ldots, x_{l+1})$ be a $\Sigma$-admissible sequence of length $l+1$. By Definition \ref{D:admissible sequence} we have $x_{l+1} = \mu_{x_l} \circ \ldots \circ \mu_{x_1}(y)$ for some $y \in \ex$ and thus -- because, by induction hypothesis, $(x_1, \ldots, x_l)$ is $(f,\Sigma,\Sigma')$-biadmissible --  $f(x_{l+1}) = f(\mu_{x_{l}} \circ \ldots \circ \mu_{x_1}(y)) = \mu_{f(x_{l})} \circ \ldots \circ \mu_{f(x_1)}(f(y))$ by axiom CM3. By axiom CM2 $f(y) \in \ex'$ so we have $f(x_{l+1}) \in \mu_{f(x_{l})} \circ \ldots \circ \mu_{f(x_1)}(ex')$ and thus $(x_1, \ldots, x_{l+1})$ is $(f, \Sigma, \Sigma')$-biadmissible.
\end{proof}

Lemma \ref{L:biadmissible} helps us to further understand ideal rooted cluster morphisms (see Definition \ref{D:ideal}): In \cite[Problem~2.12]{ADS}, Assem, Dupont and Schiffler asked for a classification of ideal rooted cluster morphisms. We provide a partial answer via the following consequence.

\begin{proposition}\label{P:without specialisation implies ideal}
Every rooted cluster morphism without specializations is ideal.
\end{proposition}

\begin{proof}
Let $\Sigma = (\XX, \ex, B)$ and $\Sigma' = (\XX', \ex', B')$ be seeds and let $f \colon \A(\Sigma) \to \A(\Sigma')$ be a rooted cluster morphism without specializations. Every element of $f(\A(\Sigma))$ can be written as an integer polynomial in the images of cluster variables of $\A(\Sigma)$. A cluster variable  $x \in \A(\Sigma)$ is of the form $x = \mu_{x_l} \circ \ldots \circ \mu_{x_1}(y)$ for $y \in \XX$ and a $\Sigma$-admissible sequence $(x_1, \ldots, x_l)$. Because $f$ is without specializations we have $f(y) \in \XX'$ and by Lemma \ref{L:biadmissible} $(f(x_1), \ldots, f(x_l))$ is $\Sigma'$-admissible. By axiom CM3 we obtain $f(x) = \mu_{f(x_l)} \circ \ldots \circ \mu_{f(x_1)}(f(y))$. This is an element of $\A(f(\Sigma))$ and thus $f(\A(\Sigma)) \subseteq \A(f(\Sigma))$. The other inclusion always holds and was proved in \cite[Lemma~2.1]{ADS}.
\end{proof}

Generally, if we have a rooted cluster morphism $f \colon \A(\Sigma) \to \A(\Sigma')$ the combinatorial structures of the two seeds $\Sigma$ and $\Sigma'$ are linked via those exchangeable cluster variables in the cluster of $\Sigma$ that do not get sent to integers. This provides a particularly strong combinatorial link between two rooted cluster algebras $\A(\Sigma)$ and $\A(\Sigma')$ for which there exists a rooted cluster morphism $f \colon \A(\Sigma) \to \A(\Sigma')$ without specializations.

\begin{lemma}\label{L:rooted cluster morphisms restrictive}
Let \[\Sigma = (\XX, \ex, B = (b_{vw})_{v,w \in \XX}) \quad \text{and} \quad \Sigma' = (\XX', \ex', B'=(b'_{v'w'})_{v',w' \in \XX'})\] be seeds and let $f \colon \A(\Sigma) \to \A(\Sigma')$ be a rooted cluster morphism. Let $x \in \ex$ be an exchangeable variable of $\Sigma$ with $f(x) \in \ex'$. 
Consider the exchangeably connected component \[f(\Sigma)^{ex}_{f(x)} = (f(\XX)^{ex}_{f(x)} \cap \XX', f(\ex)^{ex}_{f(x)}\cap \ex', (b'_{vw})_{v,w \in f(\XX)^{ex}_{f(x)}\cap \XX'})\] of $f(x)$ in the full subseed $f(\Sigma) \subseteq \Sigma'$. 
Then we have
\begin{align*}
b'_{f(y)v'} &= \sum\limits_{v \in \XX \colon f(v) = v'}b_{yv} \text{ for all $f(y) \in f(\ex)^{ex}_{f(x)} \cap \ex'$ and all $v' \in \XX'$ or } \\
b'_{f(y)v'} &= -\sum\limits_{v \in \XX \colon f(v) = v'}b_{yv} \text{ for all $f(y) \in f(\ex)^{ex}_{f(x)} \cap \ex'$ and all $v' \in \XX'$},
\end{align*}
where the empty sum is taken to be $0$. In particular, if $x,y \in \ex$ are exchangeable variables of $\Sigma$ with $f(x), f(y) \in \ex'$, then we have $b'_{f(x)f(y)} = \pm b_{xy}$. 
\end{lemma}


\begin{proof}
Let $\tilde{\XX} = \{x \in \XX \mid f(x) \in \XX'\}$ be the set of cluster variables in $\XX$ that get mapped to cluster variables in $\XX'$ and let $x \in \tilde{\XX} \cap \ex$ with $f(x) = x'$. 
Because $f$ is a ring homomorphism we have
\begin{align*}
f(\mu_x(x)x) =& f \left( \prod\limits _{v \in \XX \colon b_{xv}>0} v^{b_{xv}} + \prod\limits _{v \in \XX \colon b_{xv}<0} v^{-b_{xv}}\right) \\
=& k_1\prod\limits _{v \in \tilde{\XX} \colon b_{xv}>0} f(v)^{b_{xv}} + k_2 \prod\limits _{v \in \tilde{\XX} \colon b_{xv}<0} f(v)^{-b_{xv}}
\end{align*}
for some $k_1, k_2 \in \mathbb{Z}$. By axiom CM3 this has to be equal to
\[
\mu_{f(x)}(f(x))f(x) = \mu_{x'}(x')x' = \prod\limits _{v' \in \XX' \colon b'_{x'v'}>0} (v')^{b'_{x'v'}} + \prod\limits _{v' \in \XX' \colon b'_{x'v'}<0} (v')^{-b'_{x'v'}}.
\]
We set
\begin{align*}
M_1 = \prod\limits _{v \in \tilde{\XX} \colon b_{xv}>0} f(v)^{b_{xv}} & & M_2 = \prod\limits _{v \in \tilde{\XX} \colon b_{xv}<0} f(v)^{-b_{xv}}\\
M'_1 = \prod\limits _{v' \in \XX' \colon b'_{x'v'}>0} (v')^{b'_{x'v'}} & & M'_2 =  \prod\limits _{v' \in \XX' \colon b'_{x'v'}<0} (v')^{-b'_{x'v'}},
\end{align*}
and thus have $k_1 M_1 + k_2M_2 = M'_1 + M'_2$, where $M_1, M_2, M'_1$ and $M'_2$ are non-zero monic monomials in $\XX'$ over $\ZZ$. Assume first that $b_{x'v'} = 0$ for all $v' \in \XX'$. Then we have $M'_1 = M'_2 = 1$, which implies $k_1 M_1 + k_2M_2 = 2$ and by algebraic independence of variables in $\XX'$ we have $M_1 = M_2 = 1$. Therefore we have $b_{xv} = 0$ for all $v \in \tilde{\XX}$ and in particular for all $v' \in \XX'$ we have $0 = b'_{x'v'} = \sum_{f(v) = v'}b_{xv} = 0$. Assume now that there exists a $z' \in \XX'$ with $b'_{x'z'} \neq 0$. Without loss of generality we may assume $b'_{x'z'} > 0$. Then $z'$ divides $M'_1$. Since $z'$ does not divide $M'_2$, and $M'_1,M'_2 \neq 0$,  by algebraic independence of variables in $\XX'$ either $z$ divides $M_1$ or $z$ divides $M_2$. If $z$ divides $M_1$, by comparing coefficients of $z'$ we obtain $k_1M_1 = M'_1$ and $k_2M_2 = M'_2$ and if $z$ divides $M_2$ we obtain $k_1M_1 = M'_2$ and $k_2M_2 = M'_1$. Either way, since $M'_1$ and $M'_2$ are monic, we get 
$
k_1 = k_2 = 1
$
and the first case implies
\[b'_{f(x)v'}=b'_{x'v'} = \sum\limits_{v\in \XX \colon f(v) = v'}b_{xv},\]
for all $v' \in \XX'$ and the second case implies
\[b'_{f(x)v'}=b'_{x'v'} = -\sum\limits_{v\in \XX \colon f(v) = v'}b_{xv}\]
for all $v' \in \XX'$. In particular, if $x,y \in \ex \cap \tilde{\XX}$ then by Proposition \ref{P:almost injective} we have $b'_{f(x)f(y)} = \pm b_{xy}$.


Let now $x, y \in \tilde{\XX} \cap \ex$ be cluster variables such that their images $f(x)$ and $f(y)$ are cluster variables in the same exchangeably connected component of $f(\Sigma)$ and let them be connected by the path $f(x) = f(x_0), f(x_1), \ldots, f(x_l) = f(y)$ with $x_1, \ldots, x_{l-1} \in \ex$. Assume that we have 
\[b'_{f(x)v'} = \sum\limits_{v\in \XX \colon f(v) = v'} b_{xv}\] 
for all $v' \in \XX'$, i.e.\ no signs occur. By the above argument, this is the case if and only if $b'_{f(x)f(x_1)} = b_{xx_1}$. Iteratively applying the same argument to $x_i$ for $i = 0, \ldots, l-1$, yields that this holds  if and only if $b'_{f(x_i)f(x_{i+1})} = b_{x_ix_{i+1}}$ for all $0 \leq i < l$; in particular if and only if $b'_{f(x_{l-1})f(y)} = b_{x_{l-1}y}$, which holds if and only if \[b'_{f(y)v'} = \sum\limits_{v\in \XX \colon f(v) = v'} b_{yv}\] for all $v' \in \XX'$. This proves the claim.

\end{proof}


In the following, we want to characterize rooted cluster morphisms without specializations. Before we do that, we observe the following useful fact.

\begin{remark}\label{R:values on initial cluster}
Let $\Sigma = (\XX, \ex, B)$ and $\Sigma' = (\XX', \ex', B')$ be seeds. Any map $f \colon \XX \to \XX'$ gives rise to a unique ring homomorphism $f \colon \A(\Sigma) \to\mathcal{F}_{\Sigma'}$, because all elements of the ring $\A(\Sigma)$ are Laurent polynomials in $\XX$. Thus in particular every rooted cluster morphism without specializations is uniquely determined by its values on the cluster of the initial seed.
\end{remark}

\begin{theorem}\label{T:rooted cluster morphisms without specializations}
Let \[\Sigma = (\XX, \ex, B=(b_{vw})_{v,w \in \XX}) \quad \text{and} \quad \Sigma'= (\XX', \ex', B'=(b'_{vw})_{v,w \in \XX'})\] be seeds and let $f \colon \XX \to \XX'$ be a map. Then the algebraic extension of $f$ to $\A(\Sigma)$ gives rise to a rooted cluster morphism $f \colon \A(\Sigma) \to \A(\Sigma')$ if and only if the following hold:
\begin{itemize}
\item[$(1)$]{
The map $f$ restricts to an injection $f \mid_{\ex} \colon \ex \to \ex'$.
}
\item[$(2)$]{
If $f(x) = f(y)$ for some $x \neq y \in \XX$ then both $x$ and $y$ are coefficients of $\Sigma$. In that case for any $\Sigma$-admissible sequence $(x_1, \ldots, x_l)$, setting $\mu_{x_l} \circ \ldots \circ \mu_{x_1}(\Sigma) =: (\tilde{\XX}, \tilde{\ex}, \tilde{B} = (\tilde{b}_{vw})_{v,w \in \tilde{\XX}})$, and for any neighbour $z \in \tilde{\ex}$ of both $x$ and $y$ in $\tilde{\Sigma}$ the entries $\tilde{b}_{zx}$ and $\tilde{b}_{zy}$ have the same sign. 
}
\item[$(3)$] \label{BP3:rooted cluster morphisms without specializations}{
Let $x \in \ex$ and consider the exchangeably connected component \[f(\Sigma)^{ex}_{f(x)} = (f(\XX)^{ex}_{f(x)}, f(\ex)^{ex}_{f(x)}, (b'_{vw})_{v,w \in f(\XX)^{ex}_{f(x)}})\] of $f(x)$ in the full subseed $f(\Sigma) \subseteq \Sigma'$. 
Then we have

\begin{align*}
b'_{f(y)v'} &= \sum\limits_{v\in \XX \colon f(v) = v'}b_{yv} \text{ for all $f(y) \in f(\ex)^{ex}_{f(x)}$ and all $v' \in \XX'$ or } \\
b'_{f(y)v'} &= -\sum\limits_{v \in \XX \colon f(v) = v'}b_{yv} \text{ for all $f(y) \in f(\ex)^{ex}_{f(x)}$ and all $v' \in \XX'$},
\end{align*}
where the empty sum is taken to be $0$.
}
\end{itemize}
\end{theorem}


\begin{remark}
Condition (2) of Theorem \ref{T:rooted cluster morphisms without specializations} is not always easy to check for two given seeds $\Sigma = (\XX, \ex, B)$ and $\Sigma' = (\XX', \ex', B')$ and a map $f \colon \XX \to \XX'$. However, it is useful for checking when such a map does not induce a rooted cluster morphism. On the other hand, if for all $x, y \in \XX \setminus \ex$ with $f(x) = f(y)$ we have $b_{xv} = b_{yv}$ for all $v \in \ex$ then it is straightforward to check using the definition of matrix mutation in Definition \ref{D:mutation of a seed} that condition (2) is satisfied.
\end{remark}

\begin{proof}
Assume first that the map $f$ extends to a rooted cluster morphism. By axiom CM2 and Proposition \ref{P:almost injective} point (1) holds. By Lemma \ref{L:biadmissible} every $\Sigma$-admissible sequence is $(f, \Sigma, \Sigma')$-biadmissible. Point (2) follows from Proposition \ref{P:almost injective} by using  Proposition \ref{P:rooted cluster morphism biadmissible} to view $f$ as a rooted cluster morphism with source $\A(\tilde{\Sigma})$. By Lemma \ref{L:rooted cluster morphisms restrictive} point (3) is satisfied.

Assume, on the other hand, that $f \colon \XX \to \XX'$ is a map satisfying conditions (1) to (3). By Remark \ref{R:values on initial cluster}, it gives rise to a unique ring homomorphism $f \colon \A(\Sigma) \to \mathcal{F}_{\Sigma'}$. This ring homomorphism satisfies axioms CM1 and CM2 by definition and condition (1). It remains to check axiom CM3 and that the image $f(\A(\Sigma))$ is contained in $\A(\Sigma')$.

We show the following points for every $\Sigma$-admissible sequence $(x_1, \ldots, x_l)$ by induction on the length $l$.
\begin{itemize}
\item[(a)]{The sequence $(x_1, \ldots, x_l)$ is $(f, \Sigma, \Sigma')$-biadmissible.}
\item[(b)]{For any $y \in \XX$ we have \[f(\mu_{x_l} \circ \ldots \circ \mu_{x_1}(y)) = \mu_{f(x_l)} \circ \ldots \circ \mu_{f(x_1)}(f(y)).\]}
\item[(c)]{Set \[\mu_{x_l} \circ \ldots \mu_{x_1}(\Sigma) =: \tilde{\Sigma}  = (\tilde{\XX}, \tilde{\ex}, \tilde{B} = (\tilde{b}_{vw})_{v,w \in \tilde{\XX}})\] to be the mutation of the seed $\Sigma$ along $(x_1, \ldots, x_l)$ and \[\mu_{f(x_l)} \circ \ldots \mu_{f(x_1)}(\Sigma') =: \tilde{\Sigma}' = (\tilde{\XX}', \tilde{\ex}', \tilde{B}' = (\tilde{b}'_{vw})_{v,w \in \tilde{\XX}'})\] to be the mutation of $\Sigma'$ along $(f(x_1), \ldots, f(x_l))$. If $f(x) = f(y)$ for some $x \neq y \in \tilde{X}$, then both $x$ and $y$ are coefficients of $\tilde{\Sigma}$. (This is equivalent to saying that for any $x \in \tilde{\ex}$ and any $y \in \tilde{\XX}$ with $x \neq y$ we have $f(x) \neq f(y)$.)
}
\item[(d)]{For every $v \in \tilde{\ex}$ we have
\begin{align*}\tilde{b}'_{f(y)v'} =& \sum\limits_{v \in \tilde{\XX} \colon f(v) = v'}\tilde{b}_{yv}  &\text{for all $v' \in \tilde{\XX}'$ or } \\ \tilde{b}'_{f(y)v'} =& -\sum\limits_{v\in \tilde{\XX} \colon f(v) = v'}\tilde{b}_{yv} &\text{for all $v' \in \tilde{\XX}'$,}\end{align*}
and for all $y \in \tilde{\ex}$ such that $f(x)$ and $f(y)$ lie in the same exchangeably connected component of $f(\tilde{\Sigma})$ we have \[\tilde{b}'_{f(x)v'} = \sum\limits_{v \in \tilde{\XX} \colon f(v) = v'}\tilde{b}_{xv}\] for all $v' \in \tilde{\XX}'$ if and only if \[\tilde{b}'_{f(y)v'} = \sum\limits_{v \in \tilde{\XX} \colon f(v) = v'}\tilde{b}_{yv}\] for all $v' \in \tilde{\XX}'$.
}

\end{itemize}
If these conditions are satisfied for every $\Sigma$-admissible sequence $(x_1, \ldots, x_l)$, then by condition (b) axiom CM3 is satisfied and by conditions (a) and (b) the image of $\A(\Sigma)$ under the algebraic extension of $f$ lies in $\A(\Sigma')$. Conditions (c) and (d) are used to help prove conditions (a) and (b).

We check conditions (a) to (d) for arbitrary $\Sigma$-admissible sequences by induction on their length $l$. For a $\Sigma$-admissible sequence of length $l = 0$ conditions (a) and (b) are satisfied trivially, condition (c) is satisfied by condition (2) and condition (d) is satisfied by condition (3). Assume now that they are satisfied for all $\Sigma$-admissible sequences of length $ \leq l$ and let $(x_1, \ldots, x_{l+1})$ be a $\Sigma$-admissible sequence of length $l+1$. We set
 \[\mu_{x_l} \circ \ldots \mu_{x_1}(\Sigma) =: \tilde{\Sigma}  = (\tilde{\XX}, \tilde{\ex}, \tilde{B} = (\tilde{b}_{vw})_{v,w \in \tilde{\XX}})\]
and
 \[\mu_{f(x_l)} \circ \ldots \mu_{f(x_1)}(\Sigma') =: \tilde{\Sigma}' = (\tilde{\XX}', \tilde{\ex}', \tilde{B}' = (\tilde{b}'_{vw})_{v,w \in \tilde{\XX}'})\]
as above.

We start by proving condition (a) for $(x_1, \ldots, x_{l+1})$.  We have $x_{l+1} = \mu_{x_l} \circ \ldots \circ \mu_{x_1}(y)$ for some $y \in \ex$ and thus $f(x_{l+1}) = f(\mu_{x_{l}} \circ \ldots \circ \mu_{x_1}(y)) = \mu_{f(x_{l})} \circ \ldots \circ \mu_{f(x_1)}(f(y))$ by induction assumption (b) on the sequence $(x_1, \ldots, x_l)$. By condition (1) we have $f(y) \in \ex'$ and thus $f(x_{l+1}) \in \mu_{f(x_{l})} \circ \ldots \circ \mu_{f(x_1)}(\ex')$ and $(x_1, \ldots, x_{l+1})$ is $(f, \Sigma, \Sigma')$-biadmissible.
 
We now prove condition (b) for $(x_1, \ldots, x_{l+1})$. Let $y \in \XX$. We have $\mu_{x_l} \circ \ldots \circ \mu_{x_1}(y) \in \tilde{X}$ and $x_{l+1} \in \tilde{\ex}$. If we have $x_{l+1} \neq \mu_{x_l} \circ \ldots \circ \mu_{x_1}(y)$ this implies \[ f(\mu_{x_l} \circ \ldots \circ \mu_{x_1}(y)) \neq f(x_{l+1})\] by induction assumption (c). In this case mutation at $x_{l+1}$, respectively at $f(x_{l+1})$ acts trivially on $\mu_{x_{l}} \circ \ldots \circ \mu_{x_1}(y)$, respectively on $\mu_{f(x_{l})} \circ \ldots \circ \mu_{f(x_1)}(f(y))$ and we obtain
\begin{align*}
f(\mu_{x_{l+1}} \circ \ldots \circ \mu_{x_1}(y)) &= f(\mu_{x_{l}} \circ \ldots \circ \mu_{x_1}(y)) \\
&= \mu_{f(x_{l})} \circ \ldots \circ \mu_{f(x_1)}(f(y)) \\ &= \mu_{f(x_{l+1})} \circ \ldots \circ \mu_{f(x_1)}(f(y)),
\end{align*}
where the second equality follows from induction assumption (b) on the sequence $(x_1, \ldots, x_l)$.
If, on the other hand, $x_{l+1} = \mu_{x_l} \circ \ldots \circ \mu_{x_1}(y)$ then we have 
\begin{align}\label{E:substitute}
f(\mu_{x_{l+1}} \circ \ldots \circ \mu_{x_1}(y))
&= \frac{\prod\limits_{v \in \tilde{\XX}:\tilde{b}_{x_{l+1}v}>0}f(v)^{\tilde{b}_{x_{l+1}v}} + \prod\limits_{v\in \tilde{\XX}:\tilde{b}_{x_{l+1}v}<0}f(v)^{-\tilde{b}_{x_{l+1}v}}}{f(x_{l+1})}.
\end{align}
By induction assumption (d) we have
\begin{align*}\tilde{b}'_{f(x_{l+1})v'} =& \sum\limits_{v\in \tilde{\XX} \colon f(v) = v'}\tilde{b}_{x_{l+1}v}   \text{ for all $v' \in \tilde{\XX}'$ or } \\ \tilde{b}'_{f(x_{l+1})v'} =& -\sum\limits_{v\in \tilde{\XX} \colon f(v) = v'}\tilde{b}_{x_{l+1}v}  \text{ for all $v' \in \tilde{\XX}'$.}\end{align*}
Without loss of generality, assume that the first equation holds (otherwise we can simply change the signs below accordingly). By condition (2), for any $v' \in \tilde{\XX'}$ all non-trivial summands in \[\sum\limits_{v \in \tilde{\XX} \colon f(v) = v'}\tilde{b}_{x_{l+1}v}\] have the same sign. Therefore for any $v' \in \tilde{\XX}'$, we have $\tilde{b}_{f(x_{l+1})v'} \geq 0$ if and only if $\tilde{b}_{x_{l+1}v} \geq 0$ for all $v \in \XX$ with $f(v) = v'$, and $\tilde{b}_{f(x_{l+1})v'} = 0$ if and only if $\tilde{b}_{x_{l+1}v} = 0$ for all $v \in \tilde{\XX}$ with $f(v) = v'$. We get
\begin{align*}
\prod\limits_{v' \in \tilde{\XX}':\tilde{b}'_{f(x_{l+1})v'}>0}(v')^{\tilde{b}'_{f(x_{l+1})v'}} &= \prod\limits_{v' \in \tilde{\XX}':\tilde{b}'_{f(x_{l+1})v'}>0}(v')^{(\sum\limits_{v \in \tilde{\XX} \colon f(v) = v'}\tilde{b}_{x_{l+1}v})} \\ &= \prod\limits_{v \in \tilde{\XX}:\tilde{b}_{x_{l+1}v} >0}f(v)^{\tilde{b}_{x_{l+1}v}}
\end{align*}
and the analogous statement for the product over $v' \in \tilde{\XX}'$ with $\tilde{b}'_{f(x_{l+1})v'}<0$. Substituting into Equation (\ref{E:substitute}) we obtain that $f(\mu_{x_{l+1}} \circ \ldots \circ \mu_{x_1}(y))$ is equal to
\begin{align*}
\frac{\prod\limits_{v' \in \tilde{\XX}':\tilde{b}'_{f(x_{l+1})v'}>0}(v')^{\tilde{b}'_{f(x_{l+1})v'}} + \prod\limits_{v' \in \tilde{\XX}':\tilde{b}'_{f(x_{l+1})v'}<0}(v')^{-\tilde{b}'_{f(x_{l+1})v'}}}{f(x_{l+1})},
\end{align*}
which by definition of mutation is equal to \[\mu_{f(x_{l+1})}(f(x_{l+1})) = \mu_{f(x_{l+1})}(f(\mu_{x_l} \circ \ldots \circ \mu_{x_1}(y))).\] By induction assumption (b) we obtain \[f(\mu_{x_{l+1}} \circ \ldots \circ \mu_{x_1}(y)) = \mu_{f(x_{l+1})} \circ \ldots \circ \mu_{f(x_1)}(f(y)).\] 

We can now prove condition (c) for $(x_1, \ldots, x_{l+1})$. Let now $x \in \mu_{x_{l+1}}(\tilde{\ex})$ and $y \in \mu_{x_{l+1}}(\tilde{\XX})$ with $x \neq y$. We want to show that $f(x) \neq f(y)$. We have $x = \mu_{x_{l+1}}(\tilde{x})$ and  $y = \mu_{x_{l+1}}(\tilde{y})$ for some $\tilde{x} \in \tilde{\ex}$ and $\tilde{y} \in \tilde{\XX}$ with $\tilde{x} \neq \tilde{y}$. If both $\tilde{x} \neq x_{l+1}$ and $\tilde{y} \neq x_{l+1}$, then $x = \tilde{x} \in \tilde{\ex}$ and $y = \tilde{y} \in \tilde{\XX}$ and by induction assumption (c) we have $f(x) \neq f(y)$.
Thus assume without loss of generality that $\tilde{x}= x_{l+1}$ and $\tilde{y} \neq x_{l+1}$. Then we have 
\begin{align*}
f(x)f(x_{l+1}) &= f(\mu_{x_{l+1}}(x_{l+1}))f(x_{l+1}) \\ &= \prod\limits_{v \in \tilde{\XX}: \tilde{b}_{x_{l+1}v}>0}f(v)^{\tilde{b}_{x_{l+1}v}} + \prod\limits_{v \in \tilde{\XX}:\tilde{b}_{x_{l+1}v}<0}f(v)^{-\tilde{b}_{x_{l+1}v}}
\end{align*} 
and thus $f(x)$ divides the right hand side of the equation. On the other hand, we have $f(y) = f(\mu_{x_{l+1}}(\tilde{y})) = f(\tilde{y}) \in \tilde{\XX}'$. Assume for a contradiction that $f(x) = f(y)$. In particular, this implies $f(x) \in \tilde{\XX}'$. By algebraic independence of the elements of $\tilde{\XX}'$, $f(x)$ must divide both \[\prod\limits_{v \in \tilde{\XX}: \tilde{b}_{x_{l+1}v}>0}f(v)^{\tilde{b}_{x_{l+1}v}}\] and \[\prod\limits_{v \in \tilde{\XX}:\tilde{b}_{x_{l+1}v}<0}f(v)^{-\tilde{b}_{x_{l+1}v}}.\] This would mean that there exist $v \neq w \in \tilde{X}$ with $f(v) = f(w)$ and such that $\tilde{b}_{x_{l+1}v}>0$ and $\tilde{b}_{x_{l+1}w} < 0$, which contradicts condition (2). Thus we have $f(x) \neq f(y)$.

Finally we prove condition (d) for $(x_1, \ldots, x_{l+1})$.
Set now \[\mu_{x_{l+1}}(\tilde{B}) =: \mathcal{B} = (\beta_{vw})_{v,w \in \mu_{x_{l+1}}(\tilde{\XX})}\] and \[\mu_{f(x_{l+1})}(\tilde{B}') =: \mathcal{B}' = (\beta'_{vw})_{v,w \in \mu_{f(x_{l+1})}(\tilde{\XX}')}.\] Fix $v = \mu_{x_{l+1}}(\tilde{v}) \in \mu_{x_{l+1}}(\tilde{\ex})$. By definition of matrix mutation (Definition \ref{D:mutation of a seed}), for all $w = \mu_{x_{l+1}}(\tilde{w}) \in \mu_{x_{l+1}}(\tilde{\XX})$ we have
\begin{align*}
\beta_{vw}= \begin{cases}
-\tilde{b}_{\tilde{v}\tilde{w}}, \text{ if } \tilde{v} = x_{l+1} \text{ or } \tilde{w}= x_{l+1}\\
\tilde{b}_{\tilde{v}\tilde{w}} + \frac{1}{2}(|\tilde{b}_{\tilde{v}x_{l+1}}| \tilde{b}_{x_{l+1}\tilde{w}} + \tilde{b}_{\tilde{v}x_{l+1}} |\tilde{b}_{x_{l+1}\tilde{w}}|), \text{ else.}
\end{cases}
\end{align*}
We have shown that condition (b) holds for the sequence $(x_1, \ldots, x_{l+1})$ and thus we have $f(v) = \mu_{f(x_{l+1})}(f(\tilde{v}))$.
Thus for every $w' = \mu_{f(x_{l+1})}(\tilde{w}') \in \mu_{f(x_{l+1})}(\tilde{\XX}')$ we have
\begin{align*}
\beta'_{f(v)w'} &= \begin{cases}
-\tilde{b}'_{f(\tilde{v})\tilde{w}'}, \text{ if } f(\tilde{v}) = f(x_{l+1}) \text{ or } \tilde{w}' = f(x_{l+1})\\
\tilde{b}'_{f(\tilde{v})\tilde{w}'} + \frac{1}{2}(|\tilde{b}'_{f(\tilde{v})f(x_{l+1})}| \tilde{b}'_{f(x_{l+1})\tilde{w}'} + \tilde{b}'_{f(\tilde{v})f(x_{l+1})} |\tilde{b}'_{f(x_{l+1})\tilde{w}'}|), \\ \text{ else.}
\end{cases}
\end{align*}

By induction assumption (d) we have
\begin{eqnarray*} \label{assumption}
\tilde{b}'_{f(\tilde{v})\tilde{u}'} = \pm \sum\limits_{\tilde{u} \in \tilde{\XX} \colon f(\tilde{u})=\tilde{u}'} \tilde{b}_{\tilde{v}\tilde{u}} \text{ and } \tilde{b}'_{f(x_{l+1})\tilde{u}'} = \pm \sum\limits_{\tilde{u} \in \tilde{\XX} \colon f(\tilde{u})=\tilde{u}'}\tilde{b}_{x_{l+1}\tilde{u}}
\end{eqnarray*}
for all $\tilde{u}'\in \tilde{\XX}'$ and the signs of the two sums are the same if $f(\tilde{v})$ and $f({x_{l+1}})$ are connected by a path of variables in $f(\tilde{\ex})$, hence in particular if $\tilde{b}'_{f(v)f(x_{l+1})} \neq 0$. Note further that since $\tilde{v} \in \tilde{\ex}$ we have $\tilde{b}'_{f(\tilde{v})f(x_{l+1})} = \pm \tilde{b}_{\tilde{v}x_{l+1}}$ by induction assumption (d) and $f(\tilde{v}) = f(x_{l+1})$ if and only if $\tilde{v} = x_{l+1}$ by assumption (c). Setting \[S:=  \sum\limits_{\tilde{w} \in \tilde{\XX} \colon f(\tilde{w})=\tilde{w}'}\tilde{b}_{x_{l+1}\tilde{w}}\] we obtain 
\begin{align*}
\beta'_{f(v)w'} =& \begin{cases}
-(\pm \sum\limits_{ \tilde{w} \in \tilde{\XX} \colon f(\tilde{w})=\tilde{w}'} \tilde{b}_{\tilde{v}\tilde{w}}), \text{ if }\tilde{v} = x_{l+1} \text{ or } \tilde{w}' = f(x_{l+1}) \\
\pm \sum\limits_{\tilde{w} \in \tilde{\XX} \colon f(\tilde{w})=\tilde{w}'} \tilde{b}_{\tilde{v}\tilde{w}} + \frac{1}{2}\big(|\tilde{b}_{\tilde{v}x_{l+1}}| (\pm S) \pm \tilde{b}_{\tilde{v}x_{l+1}} | S|\big), \text{ else.}
\end{cases}
\end{align*}
Pulling out the sum yields
\begin{align*}
\beta'_{f(v)w'}=& \begin{cases}
\pm \sum\limits_{\tilde{w} \in \tilde{\XX} \colon f(\tilde{w})=\tilde{w}'} (-\tilde{b}_{\tilde{v}\tilde{w}}), \text{ if } \tilde{v} = x_{l+1}\\
\pm (-\tilde{b}_{\tilde{v}x_{l+1}}), \text{ if } \tilde{w}' = f(x_{l+1}) \\
\pm \sum\limits_{\tilde{w} \in \tilde{\XX} \colon f(\tilde{w})=\tilde{w}'} \Big( \tilde{b}_{\tilde{v}\tilde{w}} + \frac{1}{2}\big(|\tilde{b}_{\tilde{v}x_{l+1}}| \tilde{b}_{x_{l+1}\tilde{w}} + \tilde{b}_{\tilde{v}x_{l+1}} |\tilde{b}_{x_{l+1}\tilde{w}}|\big)\Big), \text{ else.}
\end{cases}\\
=&\pm \sum\limits_{\tilde{w} \in \tilde{\XX} \colon f(\tilde{w})=\tilde{w}'} \mu_{x_{l+1}}(\tilde{b}_{\tilde{v}\tilde{w}}) \\
=& \pm \sum\limits_{w \in \mu_{x_{l+1}}(\tilde{\XX}) \colon f(w) = w'} \beta_{vw},
\end{align*}
where the last equality holds because for $w = \mu_{x_{l+1}}(\tilde{w})$ by condition (b) we have \[f(w) = f(\mu_{x_{l+1}}(\tilde{w})) = \mu_{f(x_{l+1})}(f(\tilde{w})).\] 
Thus for every $w' = \mu_{f(x_{l+1})}(\tilde{w}')$ we have $f(w) = w' = \mu_{f(x_{l+1})}(\tilde{w}')$ if and only if $f(\tilde{w}) = \tilde{w}'$.

Observe that by definition of matrix mutation, if for $x,y \in \mu_{x_{l+1}}(\XX)$ with $x = \mu_{x_{l+1}}(\tilde{x})$ and $y = \mu_{x_{l+1}}(\tilde{y})$ we have $\beta'_{f(x)f(y)} \neq 0$, then we have $\tilde{b}'_{f(\tilde{x})f(\tilde{y})} \neq 0$ or both $\tilde{b}'_{f(\tilde{x})f(x_{l+1})} \neq 0$ and $\tilde{b}'_{f(\tilde{y})f(x_{l+1})} \neq 0$. Therefore, if two variables $f(x) = \mu_{f(x_{l+1})}(f(\tilde{x})) \in f(\mu_{x_{l+1}}(\XX))$ and $f(y) = \mu_{f(x_{l+1})}(f(\tilde{y})) \in f(\mu_{x_{l+1}}(\XX))$ are exchangeably connected in $f(\mu_{x_{l+1}}(\tilde{\Sigma}))$, then $f(\tilde{x})$ and $f(\tilde{y})$ are exchangeably connected in $f(\tilde{\Sigma})$. 
Thus the signs of the sums in a given exchangeably connected component of $f(\Sigma)$ carry over from $\tilde{B}'$ to $\mathcal{B}'$ and we obtain by induction assumption (d) that \begin{align*} \beta'_{f(x)u'} &= \sum\limits_{u \in \mu_{x_{l+1}}(\tilde{\XX}) \colon f(u)=u'} \beta_{xw} \end{align*} for all $u' \in \mu_{f(x_{l+1})}(\tilde{\XX}')$ if and only if \begin{align*}\beta'_{f(y)u'} &= \sum\limits_{u \in \mu_{x_{l+1} \colon f(u)=u'}(\tilde{\XX})} \beta_{yw}\end{align*}  for all $u' \in \mu_{f(x_{l+1})}(\tilde{\XX}')$.
\end{proof}

\section[Rooted cluster algebras of infinite rank as colimits of rooted cluster algebras of finite rank]{Rooted cluster algebras of infinite rank as colimits of rooted cluster algebras of finite rank
\sectionmark{Rooted cluster algebras of infinite rank}
}\label{S:colimits}
\sectionmark{Rooted cluster algebras of infinite rank}

In this section, we show that every rooted cluster algebra of infinite rank can be written as a linear
colimit of rooted cluster algebras of finite rank. This yields a formal way to manipulate cluster algebras of infinite rank by viewing them locally as cluster algebras of finite rank. 

\subsection{Colimits and limits in $Clus$}

We start by recalling the definition of colimits, and limits (their dual notion). Let $\CC$ and $\J$ be categories and let $F \colon \J \to \CC$ be a diagram of type $\J$ in the category $\CC$, i.e.\ a functor from $\J$ to $\CC$. 

%

The {\em colimit $\mathrm{colim} (F)$ of $F$} (if it exists) is an object $\mathrm{colim} (F) \in \CC$ together with a family of morphisms $f_i\colon F(i) \to \mathrm{colim} (F)$ in $\CC$ indexed by the objects $i \in \J$ such that for any morphism $f_{ij}\colon i \to j$ in $\J$ we have $f_j \circ F(f_{ij}) = f_i$ and for any object $C \in \CC$ with a family of morphisms $g_i\colon F(i) \to C$ in $\CC$ for objects $i \in \J$  such that $g_j \circ F(f_{ij}) = g_i$ for all morphisms $f_{ij}\colon i \to j$ in $\J$ there exists a unique morphism $h \colon \mathrm{colim} (F) \to C$ such that the following diagram commutes.
\[
\xymatrix{& C  & \\ & \mathrm{colim}( F) \ar[u]^h & \\ F(i) \ar[ruu]^{g_i} \ar[ru]_{f_i} \ar[rr]^{F(f_{ij})} && F(j) \ar[luu]_{g_j} \ar[lu]^{f_j}}
\]
The {\em limit $\mathrm{lim} (F)$ of $F$} (if it exists) is defined dually.

A limit $\mathrm{lim}(F)$ or colimit $\mathrm{colim}(F)$ is called {\em finite} if the index category $\J$ in the diagram $F \colon \J \to \CC$ is finite. It is called {\em small} if the index category $\J$ in the diagram $F \colon \J \to \CC$ is small.
A category is called {\em complete}, respectively {\em cocomplete}, if it has all small limits, respectively colimits.

\begin{remark}\label{R:finite limits}
Products are examples of limits. They are limits of diagrams whose index category is a discrete category, i.e.\ a category with no morphisms except the identity morphisms. Dually, coproducts are examples of colimits.

Coequalizers are examples for finite colimits. They are colimits of diagrams $G: \mathcal{J} \to \CC$, where $\mathcal{J}$ is the category with two objects $i_1$ and $i_2$ and two parallel morphisms $i_1 \rightrightarrows i_2$ in addition to the identity morphisms. Dually, equalizers are examples of finite limits.

In fact, these are rather important examples as having equalizers and small products is necessary and sufficient for a category to be complete, and dually a category is cocomplete if and only if it has coequalizers and small coproducts, see for example Mac Lane's book \cite[Chapter~V]{MacLane}.
\end{remark}

\begin{theorem}
The category $Clus$ is neither complete nor cocomplete.
\end{theorem}

\begin{proof}
If the category $Clus$ were complete, then finite products would exist, cf.\ Remark \ref{R:finite limits}. However, by  \cite[Proposition~5.4]{ADS}, the category $Clus$ does not admit finite products, hence it cannot be complete. 

Furthermore, if $Clus$ were cocomplete then coequalizers would exist. However, consider the seeds \[\Sigma_0 = (\{x_0,x_1\}, \{x_0,x_1\}, x_0 \to x_1) \text{ and } \Sigma_1 = (\{y_0,y_1\}, \{y_0,y_1\}, y_0 \to y_1)\] and the parallel rooted cluster isomorphisms defined by the algebraic extensions of
\begin{align*}
f \colon\begin{cases} \A(\Sigma_0) \to \A(\Sigma_1) \\ 
x_i \mapsto y_i \text{ for } i = 0,1 \end{cases} \text{ and }
g \colon \begin{cases} \A(\Sigma_0) \to \A(\Sigma_1) \\ 
x_i \mapsto y_{1-i}  \text{ for } i = 0,1. \end{cases}
\end{align*} 
Assume for a contradiction that there exists a coequalizer for $f$ and $g$, i.e.\;a  rooted cluster algebra $\A(\Sigma)$ with initial seed $\Sigma = (\XX, \ex, B)$ with a rooted cluster morphism $\varphi \colon \A(\Sigma_1) \to \A(\Sigma)$ such that $\varphi \circ f = \varphi \circ g$ and it is universal with this property. Because $\varphi$ is a rooted cluster morphism and $\varphi \circ f = \varphi \circ g$ we have $\varphi(y_0) = \varphi(y_1) \in \ex \cup \mathbb{Z}$. By Proposition \ref{P:almost injective} two distinct exchangeable variables of $\Sigma_1$ cannot be sent to the same exchangeable variable via a rooted cluster morphism. 
Thus we must have $\varphi(y_0) = \varphi(y_1) \in \mathbb{Z}$. Consider the empty seed $\Sigma_{\emptyset} = (\emptyset, \emptyset, \emptyset)$. As a ring, we have $\A(\Sigma_{\emptyset}) \cong \mathbb{Z}$. Consider the rooted cluster morphisms $\psi_1 \colon \A(\Sigma_1) \to \A(\Sigma_{\emptyset})$, defined by sending all cluster variables in $\A(\Sigma_1)$ to $0$, and $\psi_2 \colon \A(\Sigma_1) \to \A(\Sigma_{\emptyset})$ defined by evaluating both $y_0$ and $y_1$ at $1$. Because a rooted cluster morphism is a ring homomorphism between unital rings, any rooted cluster morphism from $\A(\Sigma)$ to $\A(\Sigma_\emptyset)$ acts as the identity on the subring $\mathbb{Z}$. Thus, if $\varphi(y_0) = \varphi(y_1) \neq 0$, then $\psi_1$ does not factor through $\varphi$ and if $\varphi(y_0) = \varphi(y_1) = 0$, then $\psi_2$ does not factor through $\varphi$. Therefore there exists no coequalizer for $f$ and $g$ and $Clus$ is not cocomplete.
\end{proof}

\subsection{Rooted cluster algebras of infinite rank as colimits} 
Even though colimits do not in general exist in $Clus$, we can show that there are sufficient colimits such that every rooted cluster algebra of infinite rank is isomorphic to a colimit of rooted cluster algebras of finite rank. More precisely, we can write them as linear colimits. A colimit $\mathrm{colim}(F)$ in a category $\CC$ is called {\em linear}, if the index category $\J$ of the diagram $F \colon \J \to \CC$ is a set endowed with a linear order viewed as a category. A diagram $F \colon \J \to \CC$ where $\J$ is endowed with a linear order $\leq$ is just a {\em linear system} of objects in $\CC$, that is a family of objects $\{C_i\}_{i \in \J}$ and a family of morphisms $\{ f_{ij}\}_{i  \leq j \in \J}$ such that $f_{jk} \circ f_{ij} = f_{ik}$ and $f_{ii} = \id_{C_i}$ for all $i\leq j \leq k$ in $\J$. In order to explicitly construct a suitable linear system of rooted cluster algebras of finite rank, we use the fact that in certain nice cases inclusions of subseeds give rise to rooted cluster morphisms.

In general, if $\Sigma$ is a full subseed of $\Sigma'$ (see Definition \ref{D:full subseed}), the natural inclusion of rings $\A(\Sigma) \to \mathcal{F}_{\Sigma'}$ does not give rise to a rooted cluster morphism $\A(\Sigma) \to \A(\Sigma')$, see \cite[Remark~4.10]{ADS}. However, we can fix this with an additional condition which has to do with how the subseed is connected to the bigger seed.

\begin{definition}
Let $\Sigma' = (\XX',\ex',B')$ be a seed with a full subseed $\Sigma = (\XX, \ex, B=(b_{vw})_{v,w \in \XX})$ such that for every $x \in \XX$ with a neighbour $y \in \XX' \setminus \XX$ in $\Sigma'$ we have $x \in \XX \setminus \ex$, i.e.\ $x$ is a coefficient of $\Sigma$. Then we say that {\em $\Sigma$ and $\Sigma'$ are connected only by coefficients of $\Sigma$}.
\end{definition}

The condition of being connected only by coefficients is transitive.

\begin{lemma}\label{L:coefficients transitive}
If $\Sigma = (\XX,\ex,B)$ is a full subseed of $\Sigma'= (\XX',\ex',B')$ and $\Sigma'$ is a full subseed of $\Sigma''= (\XX'',\ex'',B'')$, such that $\Sigma$ and $\Sigma'$ are only connected by coefficients in $\Sigma$ and such that $\Sigma'$ and $\Sigma''$ are only connected by coefficients in $\Sigma'$, then $\Sigma$ and $\Sigma''$ are only connected by coefficients in $\Sigma$.
\end{lemma}

\begin{proof}
If $x \in \ex$ is an exchangeable variable of $\Sigma$ then, by the definition of full subseed, it is an exchangeable variable of $\Sigma'$. Because $\Sigma'$ and $\Sigma''$ are only connected by coefficients of $\Sigma'$, $x$ cannot have a neighbour in $\Sigma''$ that lies in $\XX'' \setminus \XX'$. All neighbours of $x$ in $\Sigma''$ thus lie in $\XX'$, and, because $\Sigma'$ is a full subseed of $\Sigma''$, these are exactly those variables that are neighbours of $x$ in $\Sigma'$. Because $\Sigma$ and $\Sigma'$ are only connected by coefficients in $\Sigma$, these neighbours must be elements of $\XX$.
\end{proof}

If $\Sigma$ is a full subseed of $\Sigma'$, such that the seeds are connected only by coefficients of $\Sigma$, then the inclusion of $\Sigma$ in $\Sigma'$ induces a rooted cluster morphism.

\begin{lemma}\label{L:subseed morphism}
Let $\Sigma = (\XX, \ex, B)$ be a full subseed of $\Sigma' = (\XX',\ex',B')$ such that $\Sigma$ and $\Sigma'$ are connected only by coefficients of $\Sigma$. Then the inclusion $f \colon \XX \to \XX'$ gives rise to a rooted cluster morphism $f \colon \A(\Sigma) \to \A(\Sigma')$.
\end{lemma}

\begin{proof}
This follows directly from Theorem \ref{T:rooted cluster morphisms without specializations}.
\end{proof}

For any given rooted cluster algebra $\A(\Sigma)$ we can build a linear system $\{\A(\Sigma_i)\}_{i \in \ZZ}$ of rooted cluster algebras whose initial seeds are finite full subseeds $\Sigma_i$ of $\Sigma$ such that for all $i \in \ZZ$, the seeds $\Sigma_i$ and $\Sigma$ are only connected by coefficients of $\Sigma_i$. Further, we can construct it in a way, such that for all $i \leq j$ the seed $\Sigma_i$ is a full subseed of $\Sigma_j$ and the two are connected only by coefficients of $\Sigma_i$. This construction yields a linear system of rooted cluster algebras of finite rank which has the desired rooted cluster algebra $\A(\Sigma)$ as its colimit.

\begin{theorem}\label{T:connected colimit}
Every rooted cluster algebra is isomorphic to a linear colimit of rooted cluster algebras of finite rank  in the category $Clus$ of rooted cluster algebras.
\end{theorem}

\begin{proof}
Let $\A(\Sigma)$ be a rooted cluster algebra with initial seed $\Sigma = (\XX, \ex, B=(b_{vw})_{v,w \in \XX})$. Let $\Sigma = \bigsqcup_{j \in J} \Sigma^j$ be its decomposition into connected seeds with $\Sigma^j = (\XX^j, \ex^j,B^j)$ for $j \in J$, where $J$ is some countable index set (since the cluster $\XX$ is countable by Definition \ref{D:seed}, there are only countably many connected components). We can thus write the rooted cluster algebra $\A(\Sigma)$ as the countable coproduct of the connected rooted cluster algebras $\A(\Sigma^j)$:\[\A(\Sigma) \cong \coprod_{j \in J}\A(\Sigma^j).\] For notational simplicity we assume $J = \{0, 1, \ldots, n\}$ for some $n \in \ZZ_{\geq 0}$ if $J$ is finite, and $J = \ZZ_{\geq 0}$ if $J$ is infinite. We construct a linear system of rooted cluster algebras as follows. For $j \in J$ choose $x^j_0 \in \XX^j$ and inductively define full subseeds $\Sigma_i^j$ of $\Sigma$ by
\begin{align*}
\Sigma^j_0 &= (\XX^j_0, \ex^j_0,B^j_0) = (\{x^j_0\}, \emptyset, \left[0\right])\\
\Sigma^j_{i+1} &= (\XX^j_{i+1}, \ex^j_{i+1},B^j_{i+1}) \\ 
&= \left( \XX^j_i \cup \{ w \in \XX \mid b_{vw} \neq 0 \text{ for some }v \in \XX^j_i\}, \XX^j_i \cap \ex, B^j_{i+1}\right), \text{ for } i \geq 0
\end{align*}
where $B^j_{i+1}$ is the full submatrix of $B$ formed by the entries labelled by $\XX^j_{i+1} \times \XX^j_{i+1}$. Note that because $B$ is skew-symmetrizable $b_{vw} \neq 0$ is equivalent to $b_{wv} \neq 0$. Because $B^j$ is locally finite, for all $i \geq 0$ the cluster $\XX^j_i$ in the seed $\Sigma^j_i$ is finite. We set
\[
\tilde{\Sigma}_i := \coprod_{j \in J: j\leq i} \Sigma^j_{i-j} 
\]
and write $\tilde{\Sigma}_i = (\tilde{\XX}_i, \tilde{\ex}_i, \tilde{B}_i = ((\tilde{b}_i)_{vw})_{v,w \in \tilde{X}_i})$. Because the cluster in each of the seeds $\Sigma^j_{i-j}$ for $j \in J$ with $0 \leq j \leq i$ is finite, so is the cluster $\tilde{\XX}_i$ of $\tilde{\Sigma}_i$. By definition, the seed $\tilde{\Sigma}_i$ is a full subseed of the seed $\tilde{\Sigma}_{i+1}$ for all $i \geq 0$ and all the seeds $\tilde{\Sigma}_i$ are full subseeds of $\Sigma$. 

We now want to show that for all $i \geq 0$ the seeds $\tilde{\Sigma}_i$ and $\tilde{\Sigma}_{i+1}$ are connected only by coefficients of $\tilde{\Sigma}_i$. From that it follows by Lemma \ref{L:coefficients transitive}, that $\tilde{\Sigma}_i$ and $\tilde{\Sigma}_j$  for all $i \leq j$ are connected only by coefficients of $\tilde{\Sigma}_i$.
Because the subseeds $\Sigma^j_i$ and $\Sigma^{j'}_{i'}$ are by definition mutually disconnected for $j \neq j'$ in $J$ and any $i, i' \in \mathbb{Z}_{\geq 0}$, it is enough to check that $\Sigma_i^j$ and $\Sigma^j_{i+1}$ are connected only by coefficients of $\Sigma_i^j$ for any $i \in \mathbb{Z}$ and $j \in J$. Let  $x \in \ex^j_i$ and $y \in \XX^j_{i+1}$ with $b_{xy} \neq 0$. We want to show that this implies $y \in \XX^j_i$. We have $i > 0$, since $\ex^{j}_0 = \emptyset$ for all $j \in J$. It follows that $x \in \ex^j_i = \XX^j_{i-1} \cap \ex \subseteq \XX^j_{i-1}$ and thus $y \in \{ w \in \XX \mid b_{vw} \neq 0 \text{ for some }v \in \XX^j_{i-1}\} \subseteq \XX^j_i$. Therefore $\tilde{\Sigma}_i$ and $\tilde{\Sigma}_{i+1}$ are connected only by coefficients of $\tilde{\Sigma}_i$. The same argument shows that for any $i \geq 0$ the seeds $\tilde{\Sigma}_i$ and $\Sigma$ are connected only by coefficients of $\tilde{\Sigma}_i$.

By Lemma \ref{L:subseed morphism} for $0 \leq i \leq j$, the natural inclusion $f_{ij} \colon \tilde{\XX}_i \to \tilde{\XX}_j$ gives rise to a rooted cluster morphism $f_{ij} \colon \A(\tilde{\Sigma}_i) \to \A(\tilde{\Sigma}_j)$. For all $0 \leq i \leq j \leq k$ we have $f_{jk} \circ f_{ij} = f_{ik}$ and $f_{ii} = \id_{\A(\tilde{\Sigma}_i)}$, so the morphisms form a linear system of rooted cluster algebras of finite rank. Further, again by Lemma \ref{L:subseed morphism}, for $i \geq 0$ the natural inclusion $f_i \colon \tilde{\XX}_i \to \XX$ gives rise to a rooted cluster morphism $f_i \colon\A(\tilde{\Sigma}_i) \to \A(\Sigma)$. We show that $\A(\Sigma)$ together with the maps $f_i \colon \A(\tilde{\Sigma}_i) \to \A(\Sigma)$ for $i \geq 0$ is in fact the colimit of this linear system in the category of rooted cluster algebras. 

Because for any $j \in J$, the seed $\Sigma^j$ is connected, we have $\XX^j = \bigcup_{i \geq 0}\XX^j_i$ and thus 
\[\XX = \bigsqcup_{j \in J} \XX^j = \bigsqcup_{j \in J} \bigcup_{i \geq 0} \XX^j_i  = \bigcup_{i \geq 0} \bigsqcup_{j \in J} \XX^j_i = \bigcup_{i \geq 0} \tilde{\XX}_i.\] 
Because every exchange relation in $\A(\Sigma)$ lifts to an exchange relation in $\A(\tilde{\Sigma}_i)$ for all $i$ big enough (by virtue of the exchange matrices $\tilde{B}_i$ being arbitrarily large restrictions of the exchange matrix $B$), any fixed element of $\A(\Sigma)$ is contained in $\A(\tilde{\Sigma}_i)$ for all $i$ sufficiently large.

Let $\Sigma'=(\XX', \ex', Q')$ be a seed such that for all $i \geq 0$ there are rooted cluster morphisms $g_i \colon \A(\tilde{\Sigma}_i) \to \A({\Sigma'})$ compatible with the linear system $f_{ij} \colon \A(\tilde{\Sigma}_i) \to \A(\tilde{\Sigma}_j)$. 
We define a ring homomorphism $f \colon \A(\Sigma) \to \A({\Sigma'})$ by $f(x) = g_i(x)$, whenever $x \in \A(\tilde{\Sigma}_i)$, i.e.\ it is the unique ring homomorphism making the following diagram commute.

\[
\xymatrix{
&\A({\Sigma'})& \\
& \A(\Sigma) \ar@{-->}[u]^f& \\
\A(\tilde{\Sigma}_i) \ar[ruu]^{g_i} \ar[ru]_{f_i} \ar[rr]^{f_{ij}} && \A(\tilde{\Sigma}_j) \ar[luu]_{g_j} \ar[lu]^{f_j}
}
\]
For every $x \in \XX$ (respectively $x \in \ex$), there exists a $k \geq 0$ such that $x \in \tilde{\XX}_i$ (respectively $x \in \tilde{\ex}_i$) for all $i \geq k$. Thus $f(x) = g_i(x)$ for all $i \geq k$ lies in $\XX'$ (respectively in ${\ex'}$), because $g_i$ is a rooted cluster morphism for all $i \geq 0$. Thus the ring homomorphism $f$ satisfies axioms CM1 and CM2. Let now $(x_1, \ldots, x_l)$ be a $(f,\Sigma,{\Sigma'})$-biadmissible sequence and let $y \in \XX$ such that $f(y) \in {\XX'}$. Then there exists an $i \geq 0$ such that $y \in \tilde{\XX}_i$ and the sequence $(x_1, \ldots, x_l)$ is $(g_i,\tilde{\Sigma}_i,\Sigma')$-biadmissible. Thus we get \begin{align*}
f(\mu_{x_l} \circ \ldots \circ \mu_{x_1}(y)) &= f \circ f_i(\mu_{x_l} \circ \ldots \circ \mu_{x_1}(y))\\
&= g_i(\mu_{x_l} \circ \ldots \circ \mu_{x_1}(y))= \mu_{g_i(x_l)} \circ \ldots \circ \mu_{g_i(x_1)}(g_i(y)))\\
&= \mu_{f(x_l)} \circ \ldots \circ \mu_{f(x_1)}(f(y)).
\end{align*}
Therefore the ring homomorphism $f$ satisfies CM3 and is a rooted cluster morphism. Thus $\A(\Sigma)$ satisfies the required universal property.
\end{proof}

\begin{remark}
Work in progress by Stovicek and van Roosmalen shows the analogue of Theorem \ref{T:connected colimit} for cluster categories of infinite rank. However, their approach is different and it is not clear that either result can be easily obtained from the other.
\end{remark} 

\begin{remark}\label{R:colimit uncountable}
The proof of Theorem \ref{T:connected colimit} assumes that the seed of our cluster algebra has a countable cluster. We can omit this assumption, but the price we pay is that the colimit is no longer linear. If we allow seeds with uncountable clusters, by Remark \ref{R:connected implies countable} every connected component is still countable. For any given rooted cluster algebra of possibly uncountable rank, we can take the decomposition of its initial seed into (possibly uncountably many) connected components. This allows us to write our rooted cluster algebra as a (possibly uncountable) coproduct of linear colimits of rooted cluster algebras of finite rank, which -- since taking coproducts is a special example of a colimit -- is a colimit of rooted cluster algebras of finite rank.
\end{remark}

\subsection{Positivity for cluster algebras of infinite rank}\label{S:positivity}
Fomin and Zelevinsky showed in \cite[Theorem~3.1]{FZ1} that every cluster variable of a cluster algebra of finite rank is a Laurent polynomial in the elements of its initial cluster over $\ZZ$ and they conjectured that the coefficients in this Laurent polynomial are nonnegative. The so-called positivity conjecture has been a central problem in the theory of cluster algebras and has recently been solved by Lee and Schiffler \cite{LS} for all skew-symmetric cluster algebras of finite rank. Previously, the problem had been solved via different approaches for important special cases, such as for acyclic cluster algebras by Kimura and Qin \cite{KQ} and for cluster algebras from surfaces by Musiker, Schiffler and Williams \cite{MSW}.

\begin{theorem}\label{T:positivity}
The positivity conjecture holds for every skew-symmetric cluster algebra of infinite rank, i.e.\ for every skew-symmetric cluster algebra $\A(\Sigma)$ of infinite rank associated to a seed $\Sigma = (\XX, \ex, Q)$, every cluster variable in $\A(\Sigma)$ is a Laurent polynomial in $\XX$ over $\mathbb{Z}$ with nonnegative coefficients.
\end{theorem}

\begin{proof}
Let $\Sigma = (\XX, \ex, Q)$ be a skew-symmetric cluster algebra of infinite rank. Using the construction in the proof of Theorem \ref{T:connected colimit}, the associated rooted  cluster algebra $\A(\Sigma)$ can be written as a linear colimit $\A(\Sigma) = \mathrm{colim}(\A(\Sigma_i))$ of a linear system $\{\A(\Sigma_i)\}_{i \in \mathbb{Z}}$ of skew-symmetric rooted cluster algebras of finite rank with seeds $\Sigma_i = (\XX_i, \ex_i, B_i)$ and with canonical inclusions $f_i \colon \A(\Sigma_i) \to \A(\Sigma)$ for $i \in \mathbb{Z}$. Let $\tilde{x} \in \A(\Sigma)$ be a cluster variable, thus $\tilde{x} = \mu_{x_l}\circ \ldots \circ \mu_{x_1}(x)$ for some $x \in \XX$ and some $\Sigma$-admissible sequence $(x_1, \ldots, x_l)$. Then there exists an $i \in \mathbb{Z}$ such that $x \in \XX_i$ and $(x_1, \ldots, x_l)$ is $\Sigma_i$-admissible. 
Set $y = \mu_{x_l} \circ \ldots \circ \mu_{x_1}(x)$ in $\A(\Sigma_i)$. By axiom CM3 for $f_i$ we have $f_i(y) = \tilde{x}$. By \cite[Theorem~4.2]{LS}, the cluster variable $y \in \A(\Sigma_i)$ is a Laurent polynomial in $\XX_i$ over $\mathbb{Z}$ with nonnegative coefficients. Since $f_i$ is a ring homomorphism (with $f_i(1) = 1$) the image $\tilde{x}=f_i(y)$ is a Laurent polynomial in $f_i(\XX_i) \subseteq \XX$ over $\mathbb{Z}$ with nonnegative coefficients.
\end{proof}

\begin{remark}
The positivity conjecture still holds if we allow uncountable clusters: Let $\A(\Sigma)$ be a rooted cluster algebra of uncountable rank. We can decompose it into its connected components $\A(\Sigma_i)$ with seeds $\Sigma_i = (\XX_i, \ex_i,B_i)$ for $i \in I$ for some uncountable index set $I$. By Remark \ref{R:connected implies countable}, for all $i \in I$ the rooted cluster algebra $\A(\Sigma_i)$ is of countable rank and, by the defintion of coproduct, every cluster variable $x$ in $\A(\Sigma)$ lives in the cluster algebra $\A(\Sigma_i)$ of countable rank for a unique $i \in I$. Since the positivity conjecture holds for $\A(\Sigma_i)$, the cluster variable $x$ is a Laurent Polynomial in $\XX_i$ with nonnegative integer coefficients, and thus in particular a Laurent polynomial in $\XX$ with nonnegative integer coefficients.
\end{remark}

\subsection{Rooted cluster algebras from infinite triangulations of the closed disc}\label{S:triangulation colimits}

It follows from Theorem \ref{T:connected colimit} that every rooted cluster algebra arising from a countable triangulation of the closed disc can be written as a colimit of rooted cluster algebras of finite rank. Moreover, as we will see in this section, it can be written as a linear colimit of rooted cluster algebras that arise from finite triangulations of the closed disc. Thus we obtain a formal way of treating cluster algebras associated to infinite triangulations of the closed disc as infinite versions of cluster algebras of Dynkin type $A$. This provides the algebraic analogue of the work of Holm and J\o rgensen \cite{HJ} and Igusa and Todorov (\cite{IT.cluster}, \cite[Section~2.4]{IT:cyclic}), who introduced infinite versions of cluster categories of Dynkin type $A$.

\begin{theorem}\label{T:colimit of triangulations}
Let $\T$ be a countable triangulation of the closed disc with marked points $\Z$. Then the associated rooted cluster algebra $\A(\Sigma_\T)$ is isomorphic to a countable coproduct $\A(\Sigma_\T) \cong \coprod_{j \in I} \A(\Sigma_{\T_j})$ of linear colimits 
$\A(\Sigma_{\T_j}) \cong \mathrm{colim} (\A(\Sigma_{\T^j_i})$ of rooted cluster algebras $\A(\Sigma_{\T^j_i})$ of finite Dynkin type $A$.
\end{theorem}

\begin{proof}
We can directly translate the proof of Theorem \ref{T:connected colimit} to this situation. Let first $\T$ be a connected triangulation. We can build a linear system of rooted cluster algebras associated to finite triangulations of the closed disc as follows. Let $\{x_0,x_1\} \in \T$ and set $\T_0 = \{\{x_0,x_1\}\}$ and for all $i \geq 0$ set
\[
\T_{i+1} = \T_i \cup \Bigg \{\alpha \in \T \ \Big| \begin{matrix} \text{ there exists a $\beta \in \T_i$ such that $\alpha$ and $\beta$} \\ \text{ are sides of a common triangle in }\T \end{matrix}\Bigg \},
\]
where, for all $i \geq 0$, $\T_i$ as a triangulation of the closed disc with marked points $\Z_i$ being the endpoints of arcs in $\T_i$. We pass from $\T_i$ to $\T_{i+1}$ by glueing triangles to all of those edges of $\Z_i$ that are not edges of $\Z$.
We can write $\T$ as the countable union of these finite triangulations of the closed disc which are ordered by inclusion:
\[
\T = \bigcup\limits \limits_{i\geq 0} \T_i, \text{ with }\T_i \subseteq \T_j \; \text{ for all } j \geq i \geq 0.
\]
By Lemma \ref{L:subseed morphism}, the natural inclusions $f_{ij} \colon \A(\Sigma_{\T_i}) \to \A(\Sigma_{\T_j})$ for $0 \leq i \leq j$ provide a linear system of rooted cluster algebras and following the lines of the proof of Theorem \ref{T:connected colimit} there is an isomorphism of rooted cluster algebras
\[
\A(\Sigma_{\T}) \cong \mathrm{colim} (\A(\Sigma_{\T_i})).
\]
The rooted cluster algebras $\A(\Sigma_{\T_i}) $ are associated to finite triangulations of the closed disc and thus are of finite Dynkin type $A$. By Lemma \ref{L:connected triangulation}, every rooted cluster algebra associated to a triangulation of the closed disc is isomorphic to a coproduct of rooted cluster algebras associated to connected triangulations of the closed disc. This proves the claim.
\end{proof}


In the case where the set of marked points $\Z \subseteq S^1$ has precisely one limit point, the cluster algebras associated to triangulations of $\Z$ have been classified by their connected components in \cite{GG}.
%
%
This was inspired by Holm and J\o rgensen's study of the cluster category of infinite Dynkin type $A_{\infty}$. Igusa and Todorov introduced generalizations of this cluster category via the continuous cluster category of Dykin type $A$ in \cite{IT.cluster} and discrete cluster categories of Dynkin type $A$ in  \cite{IT:cyclic}. These cluster categories have combinatorial interpretations via countable triangulations of the closed disc and find their algebraic counterparts in countable coproducts of linear colimits of rooted cluster algebras of finite Dynkin type $A$.

\end{document}